\documentclass[a4paper]{amsart}

\usepackage{amsmath, amssymb, amsthm}

\usepackage{enumitem}

\usepackage{tikz-cd}

\usepackage[msc-links]{amsrefs}

\usepackage{hyperref}
\usepackage[capitalise]{cleveref}

\usepackage{hyperref}
\hypersetup{
  colorlinks   = true,
  urlcolor     = blue,
  linkcolor    = blue,
  citecolor    = blue
}

\crefname{thm}{Thm.}{}
\crefname{prop}{Prop.}{}
\crefname{lem}{Lem.}{}
\crefname{cor}{Cor.}{}
\crefname{exa}{Ex.}{}

\theoremstyle{theorem}
\newtheorem{thm}{Theorem}
\newtheorem{prop}{Proposition}
\newtheorem{lem}{Lemma}
\newtheorem{cor}{Corollary}
\newtheorem{rem}{Remark}
\newtheorem{exa}{Example}
\newtheorem{conj}{Conjecture}

 \DeclareMathOperator\Gal{Gal}
 \DeclareMathOperator\Pic{Pic}

 \newcommand{\hst}{\widehat{\h}}

\newcommand{\fd}{\mathfrak{d}}
\newcommand{\fP}{\mathfrak{P}}

\newcommand{\A}{\mathbb{A}}
\newcommand{\bP}{\mathbb{P}}
\newcommand{\WP}{\mathbb{WP}}

\newcommand{\Z}{\mathbb{Z}}
\newcommand{\Q}{\mathbb{Q}}
\newcommand{\R}{\mathbb{R}}
\newcommand{\C}{\mathbb{C}}

\newcommand{\bG}{\mathbb{G}}

\newcommand{\cO}{\mathcal{O}}
\newcommand{\cL}{\mathcal{L}}

\newcommand{\w}{\mathbf{q}}
\newcommand{\x}{\mathbf{x}}
\def\u{\mathbf{u}}

\newcommand{\h}{\mathfrak{h}}

\newcommand{\fa}{\mathfrak{a}}

\newcommand{\p}{\mathfrak{p}}
\newcommand{\lcm}{\operatorname{lcm}}

\def\wgcd{\operatorname{wgcd}}
\def\Vol{\operatorname{Vol}}
\def\Nm{\operatorname{Nm}}

\usepackage{booktabs}
\usepackage{array}

\newcommand{\Ad}{\mathcal{A}}
\newtheorem*{thm*}{Theorem}

\title[Arithmetic Sparsity and Obstructions]{Arithmetic Sparsity and Obstructions in Weighted Projective Spaces}

\author{Tanush Shaska}
\address{Department of Mathematics and Statistics, Oakland University, Rochester, MI, 48309, USA}
\email{shaska@oakland.edu}
\date{}

\begin{document}

\begin{abstract}
Let $\WP^n_{\w}$ be a weighted projective space with weights $\w = (q_0, \dots, q_n)$, $q = \lcm(q_i)$, and let $\phi \colon \WP^n_{\w} \to \bP^n$, $[x_i] \mapsto [x_i^{q/q_i}]$, be the Veronese morphism. A point of $\bP^n(\Q)$ is the image of a rational point of $\WP^n_{\w}$ only if its valuation vector at every prime satisfies a Kummer congruence. We count the rational points of $\WP^n_{\w}$ of bounded weighted height $\h = H(\phi(\,\cdot\,))^{1/q}$ and prove that, on the locus where all coordinates are nonzero,
\[
Z^{\circ}_{\h}\big( \WP^n_{\w}(\Q), X \big) = X^{q\,a(\w)} P_{\w}(\log X) + O\big( X^{q\,a(\w) - \theta} \big), \qquad \theta > 0,
\]
with $P_{\w}$ of exact degree $\beta(\w)$, where $a(\w)$ and $\beta(\w)$ are the value and the dimension of the optimal face of a linear program determined by the Kummer congruences. The exponent satisfies $Q \leq q\,a(\w) \leq q(n+1)$, $Q = \sum q_i$, with equality on the right if and only if the exponents $q/q_i$ are pairwise coprime; the difference $q(n+1) - q\,a(\w)$ measures the sparsity of the rational points of $\WP^n_{\w}$ relative to those of its Veronese image. The leading constant is evaluated when the dual optimum is diagonal and for the weights $(2,2,3,3)$. The full counting function follows by stratification, and we formulate the conjecture over an arbitrary number field.
\end{abstract}


\maketitle

\section{Introduction}

Let $\w = (q_0, \dots, q_n)$ be positive integers, $q = \lcm(q_0, \dots, q_n)$, $Q = q_0 + \cdots + q_n$, and let
\begin{equation}\label{eq-intro-veronese}
\phi \colon \WP^n_{\w} \longrightarrow \bP^n, \qquad \phi([x_0 : \cdots : x_n]) = \big[ x_0^{q/q_0} : \cdots : x_n^{q/q_n} \big],
\end{equation}
be the Veronese morphism of the weighted projective space with weights $\w$. When the exponents $n_i = q/q_i$ are pairwise coprime, as for the chart presentations $(1, q, \dots, q)$, $\phi$ is an isomorphism of varieties. In general it is finite of degree $d_\phi > 1$, and a rational point of $\bP^n$ need not be the image of a rational point of $\WP^n_{\w}$: the obstruction is a Kummer class. The moduli space of genus-two curves is $\WP^3_{(2,4,6,10)}$ in Igusa--Clebsch coordinates, and the sublocus $\cL_2$ of curves with extra automorphisms has no point of weighted height $\leq 2$ over $\Q$ \cite{2024-03}, although its image in $\bP^3$ has rational points of every small height. This paper shows that the rational points of $\WP^n_{\w}$ and of its Veronese image differ in number by a power of the height, and computes the exponent.

A weighted projective space carries two natural heights. The tautological height $\hst$ is the height of the tautological bundle on the weighted projective stack \cite{ESZB23}; on an integral representative normalized by the weighted greatest common divisor of \cite{2019-1} it is $\max_i |x_i|^{1/q_i}$, and for well-formed $\w$ Deng \cite{deng1998} proved
\begin{equation}\label{eq-intro-deng}
Z_{\hst}\big( \WP^n_{\w}(\Q), X \big) = \frac{2^n}{\zeta(Q)} X^{Q} + O\big( X^{Q - \min_i q_i} \big) ;
\end{equation}
see also \cite{Dar21, ESZB23, bruin2023}. The weighted height $\h$ of \cite{2019-1} is the normalized pullback of the Weil height, $\h(\p) = H(\phi(\p))^{1/q}$, so that its counting function on $\WP^n_{\w}(\Q)$ is the counting function of the Weil height on the image $\phi(\WP^n_{\w}(\Q)) \subseteq \bP^n(\Q)$ at the threshold $X^q$, weighted by fiber multiplicities. The two heights are related by an identity $\h = \hst / \delta$ with a defect $\delta \geq 1$ (\cref{prop-1}).

A point $y \in \bP^n(\Q)$ with nonzero coordinates lifts to $\WP^n_{\w}(\Q)$ if and only if, for a primitive integral representative and every prime $p$, the valuation vector $(v_p(y_0), \dots, v_p(y_n))$ lies in the lifting set
\[
M_{\w} = \left\{ t \in \Z^{n+1}_{\geq 0} : \min_i t_i = 0, \ t_i \equiv t_j \!\!\pmod{\gcd(n_i, n_j)} \right\},
\]
together with a sign condition (\cref{prop:lifting-criterion}). Let $T_{\w}$ be the finite set of minimal nonzero elements of $M_{\w}$, and let $a(\w)$ and $\beta(\w)$ be the value and the dimension of the optimal face of the linear program
\[
\max \sum_{g \in T_{\w}} c_g \qquad \text{subject to} \qquad c_g \geq 0, \quad \sum_{g \in T_{\w}} c_g \, g_i \leq 1 \quad (0 \leq i \leq n) .
\]

\begin{thm*}
Let $n \geq 1$ and $\gcd(q_0, \dots, q_n) = 1$. There are $\theta > 0$ and a polynomial $P_{\w}$ of exact degree $\beta(\w)$ with positive leading coefficient such that
\[
Z^{\circ}_{\h}\big( \WP^n_{\w}(\Q), X \big) = X^{q \, a(\w)} P_{\w}(\log X) + O\big( X^{q\,a(\w) - \theta} \big),
\]
where $Z^{\circ}_{\h}$ counts the points with all coordinates nonzero.
\end{thm*}

This is \cref{thm:main}. The full counting function follows by stratification along the coordinate strata, which are again weighted projective spaces, and a proper stratum may dominate (\cref{cor:full-count}). The exponent satisfies $Q \leq q\,a(\w) \leq q(n+1)$, with equality on the right if and only if the $n_i$ are pairwise coprime and on the left if and only if $\langle t, \w \rangle \geq q$ for every nonzero $t \in M_{\w}$ (\cref{prop:a-bounds}); by Schanuel's theorem the ratio of $Z^{\circ}_{\h}$ to the number of points of $\bP^n(\Q)$ of height $\leq X^q$ is $\asymp X^{-q(n+1-a(\w))} (\log X)^{\beta(\w)}$ (\cref{cor:sparsity}). For $\w = (2,3,5)$ one has $q\,a(\w) = 12$ against $Q = 10$ and $q(n+1) = 90$.

The analytic input is the multivariable Tauberian theorem of de la Bret\`eche \cite{delaBreteche2001}, which bounds the degree of $P_{\w}$ from above; the exact degree is obtained from an unconditional lower bound proved by parametrizing admissible tuples along the primal optimal face (\cref{lem:lower}), so no nondegeneracy hypothesis is needed. The leading constant is computed in two cases: as an Euler product when the only tight dual constraints are the axis ones (\cref{thm:diagonal}), which includes $(1,2,3,5)$, and by an elementary argument for $(2,2,3,3)$ (\cref{thm:2233}), where the polar family does not have full rank. Exact values of the counts for several weights are compared with the theorem in \cref{sec-numerics}.

On the interior of $\WP^n_{\w}$ the Galois-theoretic field of definition of a point coincides with the minimal field over which it has a representative (\cref{prop:fod}), so the points counted here are the $\Q$-points of the variety. \cref{sec-7} formulates the conjecture over a number field and the problem for points of fixed degree.

\section{Preliminaries}\label{sec-2}

Let $k$ be a number field of degree $m = [k:\Q]$ with algebraic closure $\bar k$, let $M_k$ be its set of places, $M_k^0$ the finite and $M_k^\infty$ the infinite ones, and for $v \in M_k$ let $|\cdot|_v$ be the normalized absolute value and $n_v = [k_v : \Q_v]$ the local degree. The weighted projective space $\WP^n_{\w}(\bar k)$ with weights $\w = (q_0, \dots, q_n) \in \Z_{\geq 1}^{n+1}$ is the set of points $\p = [x_0 : \cdots : x_n]$, $x_i \in \bar k$ not all zero, modulo $\p \sim \lambda \star \p = [\lambda^{q_0} x_0 : \cdots : \lambda^{q_n} x_n]$ for $\lambda \in \bar k^\times$. Throughout, $q = \lcm(q_0, \dots, q_n)$, $Q = q_0 + \cdots + q_n$, $n_i = q/q_i$, and $\WP^n_{\w}(k)$ denotes the set of points admitting a representative $\x \in k^{n+1}$. For $\p \in \WP^n_{\w}(\bar k)$ we write $k(\phi(\p))$ for the field of definition of its image under the Veronese morphism \cref{Veronese}, the field generated over $k$ by the ratios $x_i^{n_i}/x_j^{n_j}$ with $x_j \neq 0$; the field of definition of $\p$ itself is discussed in \cref{subsec-fod}.

For $P = [z_0 : \cdots : z_n] \in \bP^n(k)$ the absolute Weil height is
\begin{equation}\label{eq-weil}
H(P) = \prod_{v \in M_k} \max \{ |z_0|_v^{n_v}, \ldots, |z_n|_v^{n_v} \}^{1/[k:\Q]} ,
\end{equation}
and Schanuel's theorem \cite{schanuel1979} states that $Z_H(\bP^n(k), X) = \#\{ P \in \bP^n(k) : H(P) \leq X \} = S_k(n) X^{m(n+1)} + O(X^{m(n+1)-1} \log X)$ with
\begin{equation}\label{eq-schanuel-const}
S_k(n) = \frac{h_k R_k}{w_k \zeta_k(n+1)} \left( \frac{2^{r_k} (2\pi)^{s_k}}{\sqrt{|\Delta_k|}} \right)^{n+1} (n+1)^{r_k + s_k - 1},
\end{equation}
$h_k, R_k, w_k, \zeta_k, \Delta_k$ the class number, regulator, number of roots of unity, zeta function and discriminant, and $r_k, s_k$ the numbers of real and complex places.

\subsection{The Veronese morphism}

The morphism
\begin{equation}\label{Veronese}
\phi \colon \WP^n_{\w} \to \bP^n, \qquad \phi([x_0 : \cdots : x_n]) = [x_0^{n_0} : \cdots : x_n^{n_n}],
\end{equation}
is finite and dominant. Its degree $d_\phi$ is the degree of the function field extension $[k(\WP^n_{\w}) : \phi^* k(\bP^n)]$, equivalently the index of the graded inclusion $k[y_0, \dots, y_n] \hookrightarrow k[x_0, \dots, x_n]$, $y_i \mapsto x_i^{n_i}$, in the sense of Veronese subrings.

\begin{lem}\label{lem:d_phi_definition}
Let $d = \gcd(q_0, \dots, q_n)$. Every $y \in \bP^n(\bar k)$ with all coordinates nonzero has exactly
\[
d_\phi = \frac{q^n d}{\prod_{i=0}^n q_i}
\]
preimages in $\WP^n_{\w}(\bar k)$, and $d_\phi$ is the degree of $\phi$.
\end{lem}

\begin{proof}
A preimage $\p = [x_0 : \cdots : x_n]$ of $y$ satisfies $x_i^{n_i} = \lambda y_i$ for some $\lambda \in \bar k^\times$; replacing $\p$ by $\mu \star \p$ scales $\lambda$ by $\mu^q$, so we may fix $\lambda = 1$. Then $x_i = \zeta_i y_i^{1/n_i}$ for a fixed choice of roots and $\zeta_i \in \mu_{n_i}$, giving $\prod_i n_i$ tuples. Two such tuples represent the same point if and only if they differ by $\mu \star$ with $\mu^q = 1$, since the image must stay $[y_0 : \cdots : y_n]$ with $\lambda = 1$, and $\mu \in \mu_q$ acts on the tuples by $\zeta_i \mapsto \mu^{q_i} \zeta_i$ with stabilizer $\{ \mu : \mu^{q_i} = 1 \ \forall i \} = \mu_d$. Hence the fiber has $\prod_i n_i / (q/d) = q^n d / \prod_i q_i$ points. As $\phi$ is finite and generically \'etale in characteristic zero, this is its degree.
\end{proof}

\subsection{The weighted height}

For $\p = [x_0 : \cdots : x_n] \in \WP^n_{\w}(k)$ the weighted height of \cite{2019-1} is
\begin{equation}\label{eq-wheight}
\h_k(\p) = \prod_{v \in M_k} \max \left\{ |x_0|_v^{n_v / q_0}, \ldots, |x_n|_v^{n_v / q_n} \right\} .
\end{equation}
By \cite{2019-1}, $\h_k$ is independent of the representative, satisfies $\h_k(\p) \geq 1$, is invariant under $\Gal(\bar k / k)$, satisfies $\h_L(\p) = \h_k(\p)^{[L:k]}$ for finite $L/k$, and
\begin{equation}\label{eq-h-pullback}
\h_k(\p) = H(\phi(\p))^{[k:\Q]/q} ;
\end{equation}
the absolute weighted height $\h(\p) = \h_K(\p)^{1/[K:\Q]}$, $K$ any field over which $\p$ has a representative, is therefore $\h(\p) = H(\phi(\p))^{1/q}$, the pullback of the Weil height along $\phi$ normalized by $1/q$. Local and global versions are developed in \cite{2022-1}, and \cite{2023-1} proves a form of Vojta's conjecture for $\h$. The set $\{ \p \in \WP^n_{\w}(\bar k) : \h(\p) \leq c_0, \ [k(\phi(\p)) : k] \leq d_0 \}$ is finite \cite{2019-1}: by \cref{eq-h-pullback} it maps into $\{ P \in \bP^n(\bar\Q) : H(P) \leq c_0^q, \ [\Q(P):\Q] \leq m d_0 \}$, which is finite by Northcott's theorem, and $\phi$ is finite-to-one by \cref{lem:d_phi_definition}.

The weighted greatest common divisor of \cite{2019-1} is, for $\x \in \Z^{n+1}$, the largest $d \in \Z_{\geq 1}$ with $d^{q_i} \mid x_i$ for all $i$; an integral representative with $\wgcd(\x) = 1$ is called normalized, and by \cite[Lemma~7]{2019-1} it is unique up to $q$-th roots of unity for well-formed weights, those with $\gcd(q_j : j \neq i) = 1$ for every $i$. Over a number field with nontrivial class group there is no canonical generator to normalize by, and the ideal-theoretic substitute is the content ideal of the next section.

\subsection{The tautological height and the defect}

The weighted height $\h_k$ of \cref{eq-wheight} is one of two natural heights carried by a weighted projective space. In this section we introduce the \emph{tautological height} $\hst_k$: the adelic height whose finite local factors are governed by the weighted greatest common divisor and which, over $\Q$ and in normalized coordinates, is simply $\max_i |x_i|^{1/q_i}$. In the stack-theoretic frameworks of \cite{ESZB23, Dar21}, $\hst$ is the height attached to the tautological bundle on the weighted projective stack with weights $\w$, generalizing the naive height of an elliptic curve, which is the case $\w = (4,6)$; see \cref{rem-stack}. The two heights differ by a \emph{defect} $\delta_k(\p) \geq 1$ measuring the fractional parts of the weighted valuations of the coordinates, and the distribution of this defect is responsible for the different growth rates of the counting functions of $\hst$ and $\h$ compared in \cref{sec-main}.

For $v \in M_k^0$ we denote by $\fP_v$ the corresponding prime ideal of $\cO_k$ and normalize the absolute values so that $|x|_v^{n_v} = \Nm(\fP_v)^{-v(x)}$, where $v(\cdot)$ is the valuation with $v(k^\times) = \Z$ and $\Nm$ denotes the ideal norm. For a tuple $\x = (x_0, \dots, x_n) \in k^{n+1}$, not all coordinates zero, and $v \in M_k^0$ we set
\begin{equation}\label{eq-theta}
\theta_v(\x) = \min_{x_i \neq 0} \frac{v(x_i)}{q_i} \in \Q .
\end{equation}
The \emph{weighted content ideal} of $\x$ is the fractional ideal
\begin{equation}\label{eq-content}
\fd_k(\x) = \prod_{v \in M_k^0} \fP_v^{\lfloor \theta_v(\x) \rfloor} .
\end{equation}
Since $\lfloor \min_i a_i \rfloor = \min_i \lfloor a_i \rfloor$ for any real numbers $a_i$, we have equivalently $\fd_k(\x) = \prod_{v} \fP_v^{\min_i \lfloor v(x_i)/q_i \rfloor}$. Hence, for integral tuples, $\fd_k(\x)$ is the ideal-theoretic form of the weighted greatest common divisor of \cite{2019-1}; over $\Q$, or over any $k$ of class number one, it is the ideal generated by $\wgcd(\x)$. Two remarks are in order: 

First, \cref{eq-content} requires no generator: only the ideal $\fd_k(\x)$ and its norm enter the sequel, so the definition is unambiguous over every number field, including those with nontrivial class group. 

Second, \cref{eq-content} extends the $\wgcd$ from integral tuples to arbitrary tuples in $k^{n+1}$, the ideal $\fd_k(\x)$ being fractional when negative valuations occur; see \cref{exa-2}.

The \emph{tautological height} of $\p = [x_0 : \dots : x_n] \in \WP^n_{\w}(k)$ is defined as
\begin{equation}\label{eq-hst}
\begin{split}
\hst_k(\p) & = \Nm \left( \fd_k(\x) \right)^{-1} \cdot \prod_{v \in M_k^\infty} \max_{0 \leq i \leq n} |x_i|_v^{n_v/q_i} \\
           & = \prod_{v \in M_k^0} \Nm(\fP_v)^{-\lfloor \theta_v(\x) \rfloor} \cdot \prod_{v \in M_k^\infty} \max_{0 \leq i \leq n} |x_i|_v^{n_v/q_i} .
\end{split}
\end{equation}
Thus $\hst_k$ has exactly the shape of the weighted height $\h_k$, except that at each finite place the exponent $\theta_v$ is replaced by its integer part. The next lemma shows that this is a well-defined, positive height on $\WP^n_{\w}(k)$.

\begin{lem}\label{lem-hst}
Let $\p \in \WP^n_{\w}(k)$. The following hold:
\begin{enumerate}[label=\roman*)]
\item $\hst_k(\p)$ does not depend on the choice of weighted homogeneous coordinates for $\p$;

\item $\hst_k(\p) \geq \h_k(\p) \geq 1$;

\item if $k = \Q$ and $\x \in \Z^{n+1}$ is a representative with $\wgcd(\x) = 1$, then $\hst(\p) = \max_{0 \leq i \leq n} |x_i|^{1/q_i}$.
\end{enumerate}
\end{lem}

\begin{proof}
i) Replace $\x$ by $\lambda \star \x = (\lambda^{q_0} x_0, \dots, \lambda^{q_n} x_n)$ for $\lambda \in k^\times$. At a finite place, $v(\lambda^{q_i} x_i)/q_i = v(\lambda) + v(x_i)/q_i$, hence $\theta_v(\lambda \star \x) = \theta_v(\x) + v(\lambda)$. Since $v(\lambda) \in \Z$, we get $\lfloor \theta_v(\lambda \star \x) \rfloor = \lfloor \theta_v(\x) \rfloor + v(\lambda)$, so $\fd_k(\lambda \star \x) = \fd_k(\x) \cdot (\lambda)$ and
\[
\Nm \left( \fd_k(\lambda \star \x) \right)^{-1} = \Nm \left( \fd_k(\x) \right)^{-1} \prod_{v \in M_k^0} |\lambda|_v^{n_v} .
\]
At the Archimedean places the maximum scales by $\prod_{v \in M_k^\infty} |\lambda|_v^{n_v}$, and the product formula $\prod_{v \in M_k} |\lambda|_v^{n_v} = 1$ gives $\hst_k(\lambda \star \x) = \hst_k(\x)$.

ii) At every finite place $\lfloor \theta_v \rfloor \leq \theta_v$, hence $\Nm(\fP_v)^{-\lfloor \theta_v \rfloor} \geq \Nm(\fP_v)^{-\theta_v}$, and the Archimedean factors of $\hst_k$ and $\h_k$ coincide; thus $\hst_k(\p) \geq \h_k(\p)$, and $\h_k(\p) \geq 1$ by \cite{2019-1}.

iii) If $\x$ is integral with $\wgcd(\x) = 1$, then $\min_i \lfloor v(x_i)/q_i \rfloor = 0$ at every finite place, so $\fd_\Q(\x) = (1)$ and only the Archimedean factor survives.
\end{proof}

The \emph{defect} of $\p \in \WP^n_{\w}(k)$ is defined as
\begin{equation}\label{eq-defect}
\delta_k(\p) = \prod_{v \in M_k^0} \Nm(\fP_v)^{\theta_v(\x) - \lfloor \theta_v(\x) \rfloor} \in [1, \infty) .
\end{equation}
The exponent $\theta_v - \lfloor \theta_v \rfloor \in [0,1) \cap \Q$ is unchanged under $\x \mapsto \lambda \star \x$, since the shift $v(\lambda)$ is an integer, so $\delta_k(\p)$ is an invariant of the point; the product is finite since $\theta_v = 0$ for all but finitely many $v$.

\begin{prop}\label{prop-1}
For every $\p \in \WP^n_{\w}(k)$,
\[
\h_k(\p) = \frac{\hst_k(\p)}{\delta_k(\p)} .
\]
In particular $\h_k(\p) \leq \hst_k(\p)$, with equality if and only if $\theta_v(\x) \in \Z$ for every finite place $v$. Over $k = \Q$, for an integral representative $\x$ with $\wgcd(\x) = 1$, equality holds if and only if $\gcd(x_0, \dots, x_n) = 1$, where $\gcd$ denotes the ordinary greatest common divisor.
\end{prop}

\begin{proof}
Fix a finite place $v \in M_k^0$. By the normalization $|x_i|_v^{n_v} = \Nm(\fP_v)^{-v(x_i)}$, the local factor of $\h_k$ at $v$ is
\[
\max_{0 \leq i \leq n} |x_i|_v^{n_v/q_i} = \max_i \Nm(\fP_v)^{-v(x_i)/q_i} = \Nm(\fP_v)^{-\min_i v(x_i)/q_i} = \Nm(\fP_v)^{-\theta_v(\x)} ,
\]
the middle equality because $\Nm(\fP_v) > 1$, so the maximum is attained at the smallest exponent. The corresponding factor of $\hst_k$ replaces $\theta_v$ by $\lfloor \theta_v \rfloor$ by \cref{eq-hst}, while the Archimedean factors of $\h_k$ and $\hst_k$ coincide. Hence
\[
\frac{\hst_k(\p)}{\h_k(\p)} = \prod_{v \in M_k^0} \Nm(\fP_v)^{\theta_v(\x) - \lfloor \theta_v(\x) \rfloor} = \delta_k(\p) ,
\]
which is the identity. Since $\theta_v - \lfloor \theta_v \rfloor \in [0,1)$ and $\Nm(\fP_v) > 1$, every factor is $\geq 1$, so $\h_k(\p) \leq \hst_k(\p)$, with equality if and only if $\theta_v(\x) \in \Z$ at every finite place.

For the last statement, if $\x$ is integral and normalized then $\theta_v \in [0,1)$ for every $v = v_p$, so $\theta_v \in \Z$ if and only if $\theta_v = 0$, if and only if $\min_i v_p(x_i) = 0$, if and only if $p \nmid \gcd(x_0, \dots, x_n)$.
\end{proof}

Since $\hst \geq \h$, the finiteness statement of \cref{sec-2} holds for $\hst$ as well.

We illustrate the two heights and the defect.

\begin{exa}\label{exa-2}
Let $\p = [1 : 1/3 : 1 : 1] \in \WP^3_{(2,4,6,10)}(\Q)$. At $p = 3$ the valuations are $v_3(x_i) = 0, -1, 0, 0$, so
\[
\theta_3 = \min \left( 0, -\tfrac{1}{4}, 0, 0 \right) = -\tfrac{1}{4}, \qquad \lfloor \theta_3 \rfloor = -1, \qquad \fd(\p) = (3)^{-1},
\]
a fractional content ideal. A normalized integral representative is obtained by scaling with $\lambda = 3$, namely $3 \star \p = [3^2 : 3^3 : 3^6 : 3^{10}] = [9 : 27 : 729 : 59049]$, for which $\min_i \lfloor v_3(x_i)/q_i \rfloor = \min(1, 0, 1, 1) = 0$, so $\wgcd = 1$. The tautological height may be computed from either representative:
\[
\begin{split}
\hst(\p) & = \Nm \left( (3)^{-1} \right)^{-1} \cdot \max \left( 1, 3^{-1/4}, 1, 1 \right) = 3 \cdot 1 = 3 , \\
\hst(\p) & = \max \left( 9^{1/2}, 27^{1/4}, 729^{1/6}, 59049^{1/10} \right) = 3 .
\end{split}
\]
The defect is $\delta(\p) = 3^{\theta_3 - \lfloor \theta_3 \rfloor} = 3^{3/4}$, and indeed
\[
\h(\p) = \frac{\hst(\p)}{\delta(\p)} = \frac{3}{3^{3/4}} = 3^{1/4} ,
\]
in agreement with the direct computation of $\h(\p)$, whose factor at $v = 3$ is
\[
\max \left( 1, |1/3|_3^{1/4}, 1, 1 \right) = 3^{1/4}
\]
and whose Archimedean factor is $1$. Note that the normalized representative has $\gcd(9, 27, 729, 59049) = 9 \neq 1$, consistent with \cref{prop-1}.
\qed
\end{exa}

\begin{rem}\label{rem-base}
The two heights behave differently under base change. By \cite{2019-1}, $\h_L(\p) = \h_k(\p)^{[L:k]}$ for any finite extension $L/k$, whereas $\hst$ fails this relation, since ramification changes the integer parts $\lfloor \theta_v \rfloor$. 

Take $\p = [2 : 2] \in \WP^1_{(1,2)}(\Q)$: here $\theta_2 = \min(1, \tfrac{1}{2}) = \tfrac{1}{2}$, so $\delta_\Q(\p) = 2^{1/2}$, $\hst_\Q(\p) = 2$, and $\h_\Q(\p) = 2^{1/2}$. 

Over $L = \Q(\sqrt{2})$, with $\fP = (\sqrt{2})$, one has $\theta_\fP = \min(2, 1) = 1 \in \Z$, so $\delta_L(\p) = 1$ and
\[
\hst_L(\p) = \h_L(\p) = \h_\Q(\p)^2 = 2 \neq \hst_\Q(\p)^2 = 4 .
\]
The defect thus measures the local weighted structure of $\p$, visible over $k$ and trivializable over suitable ramified extensions; it is the   absolute normalization of \cite{2019-1}.
\end{rem}

\begin{rem}\label{rem-stack}
The height $\hst$ is attached to the tautological bundle $\cO(1)$ on the weighted projective stack with weights $\w$, in the sense of \cite{ESZB23}. The two definitions agree exactly at the finite places, since $\lceil -\theta_v \rceil = -\lfloor \theta_v \rfloor$; at the Archimedean places the height of \cite{ESZB23} depends on a choice of metric and is defined only up to a bounded function. Under this identification $\h$ is their stable height and $\log \delta_k$ the sum of their local discrepancies, so \cref{prop-1} is the multiplicative form of the decomposition of a stacky height into stable part and local discrepancies. 

For $\w = (4,6)$ the height $\hst$ is the naive height $\max(|A|^{1/4}, |B|^{1/6})$ of an elliptic curve $y^2 = x^3 + Ax + B$; over $\Q$ it goes back to \cite{deng1998} and \cite{2019-1}, as noted in \cite{ESZB23}.
\end{rem}

\subsection{Related Work on Counting Rational Points}

The counting problem for $\hst$ goes back to \cite{deng1998}, where an asymptotic is obtained over an arbitrary number field with respect to a size function. Over $\Q$, for a well-formed weighted projective space with $n \geq 2$, it reads
\begin{equation}\label{eq-deng}
Z_{\hst}(\WP^n_{\w}(\Q), T) = \frac{2^n}{\zeta(Q)} T^Q + O\left(T^{Q - \min_i q_i}\right),
\end{equation}
where $Q = q_0 + \cdots + q_n$; for $\w = (1, \dots, 1)$ this is Schanuel's theorem \cite{schanuel1979} for $\bP^n(\Q)$. Over $\Q$ the proof is a lattice-point count on integral representatives normalized by $\wgcd(\x) = 1$, on which $\hst = \max_i |x_i|^{1/q_i}$ by \cref{lem-hst}~iii); such a representative is unique up to sign when all coordinates are nonzero and $\gcd(q_0, \dots, q_n) = 1$, since $\lambda \star \x = \x'$ with $\x, \x'$ both normalized forces $\lambda^{q_i} \in \Q^\times$ for all $i$, hence $\lambda \in \Q^\times$ and $(\lambda) = (1)$ by the proof of \cref{lem-hst}~i). Deng assumes $\WP^n_{\w}$ well-formed; the assumption enters his argument only through $\gcd(q_0, \dots, q_n) = 1$ \cite[Lemma~3.1]{deng1998}, which is what makes the scalar relating two $k$-representatives of a point lie in $k$. Over a number field $k$ of degree $m$, with points counted by $\hst_k \leq T$, \cite[Theorem~(A)]{deng1998} gives $N \sim C^{\w}_k T^{Q}$ with
\[
C^{\w}_k = \frac{h_k R_k}{w_k \zeta_k(Q)} \Big( \frac{2^{r_k} (2\pi)^{s_k}}{\sqrt{|\Delta_k|}} \Big)^{n+1} Q^{r_k + s_k - 1}
\]
and error term $O(T^{Q - \min_i q_i / m})$; this differs from Schanuel's $S_k(n)$ in \cref{eq-schanuel-const} in exactly two places, $\zeta_k(Q)$ in place of $\zeta_k(n+1)$ and $Q^{r_k+s_k-1}$ in place of $(n+1)^{r_k+s_k-1}$. Darda \cite{Dar21} recovered the same asymptotic, without error terms, by the method of height zeta functions, within a general treatment of heights on weighted projective stacks; in \cite{ESZB23} the exponent $Q$ appears as an instance of a Batyrev--Manin--Malle conjecture for stacky heights.

Over a number field $K$, the authors of \cite{bruin2023} count points $x \in \WP^n_{\w}(K)$ by the height $\hst(\phi(x))$ of their image under a morphism $\phi \colon \WP^n_{\w} \to \WP^{n'}_{\u}$ homogeneous of degree $e$. Their asymptotic \cite[Thm.~3.15]{bruin2023} is
\[
N_\phi(T) = C T^{Q/e} + O\left(T^{Q/e - \min_i \frac{q_i}{e [K:\Q]}}\right),
\]
where the constant $C$ involves the Dedekind zeta value $\zeta_K(Q)$, the discriminant and roots of unity of $K$, a volume $\Vol_{\Nm}(B_\phi(1))$, and a local factor $C_K^\phi$ given by a sum over ideals. For $K = \Q$ and $\phi$ the identity this recovers the asymptotic of \cite{deng1998}.

These results concern $\hst$. The two heights are heights for the same $\Q$-line bundle $\cO(1)$ and differ only at the finite places, $\hst$ being the height of the tautological bundle in the sense of \cite{ESZB23} and $\h$ its stable height (\cref{rem-stack}); the counting problem for $\h$ is treated in \cref{sec-main}, and \cref{prop-1} connects it to \cref{eq-deng}.

\section{The chart family}\label{sec-3}

For presentations with pairwise coprime exponents $n_i = q/q_i$ the morphism $\phi$ is an isomorphism of varieties and the count for $\h$ is Schanuel's count at the threshold $X^q$; this is the case $q\,a(\w) = q(n+1)$ of \cref{prop:a-bounds}. Over a number field $k$ of degree $m$ we use the absolute normalization: $Z_{\h}(\WP^n_{\w}(k), X)$ counts the points $\p \in \WP^n_{\w}(k)$ with $\h_k(\p) \leq X^m$, so that for $k = \Q$ the two normalizations agree, exactly as for the Weil height \cref{eq-weil}.

By the pullback identity \cref{eq-h-pullback}, the condition $\h(\p) \leq X$ is the condition $H(\phi(\p)) \leq X^q$ on the image, so $Z_{\h}$ counts rational points of $\bP^n(\Q)$ of Weil height at most $X^q$ lying in the image of $\phi$, weighted by the number of their rational preimages. The image is a proper subset of $\bP^n(\Q)$ in general, governed by the local solvability of the system $x_i^{q/q_i} = \lambda y_i$, which we study in \cref{sec-5}; here we treat the family of weights for which $\phi$ is a bijection on rational points and the asymptotic follows from Schanuel's theorem alone.

\begin{prop}\label{prop-hcount}
Let $n \geq 1$, $q \geq 2$, and $\w = (1, q, \dots, q)$. Then
\[
Z_{\h}(\WP^n_{\w}(\Q), X) = \frac{2^n}{\zeta(n+1)} X^{q(n+1)} + O\left( X^{qn} \log X \right).
\]
\end{prop}

\begin{proof}
Since $q_0 = 1$, every point with $x_0 \neq 0$ has a unique representative of the form $[1 : u_1 : \cdots : u_n]$, obtained by acting with $\lambda = 1/x_0$; and the representative is unique because $\mu \star [1 : \u] = [1 : \u']$ forces $\mu^{q_0} = \mu = 1$. This identifies $\{ \p : x_0 \neq 0 \}$ with $\Q^n$. For such a point,
\[
\h(\p) = \prod_{v \in M_\Q} \max \left( 1, |u_1|_v^{1/q}, \dots, |u_n|_v^{1/q} \right) = H_{\mathrm{aff}}(\u)^{1/q},
\]
where $H_{\mathrm{aff}}$ is the affine Weil height on $\Q^n$. Hence the points of the chart with $\h \leq X$ are the $\u \in \Q^n$ with $H_{\mathrm{aff}}(\u) \leq X^q$, and by Schanuel's theorem applied to $\bP^n(\Q)$ and the hyperplane $y_0 = 0$,
\[
\# \{ \u \in \Q^n : H_{\mathrm{aff}}(\u) \leq B \} = \frac{2^n}{\zeta(n+1)} B^{n+1} + O\left( B^{n} \log B \right).
\]
The remaining points form the locus $x_0 = 0$, which is $\WP^{n-1}_{(q, \dots, q)} \cong \bP^{n-1}$ with $\h = H^{1/q}$, contributing $O(X^{qn})$. Setting $B = X^q$ gives the result.
\end{proof}

\begin{exa}\label{exa-3}
For $\w = (1,2)$, so $Q = 3$ and $q = 2$, the two counting functions grow at different rates:
\[
\begin{split}
Z_{\hst}(\WP^1_{(1,2)}(\Q), X) 	& 	= \frac{2}{\zeta(3)} X^{3} + O(X^2), \\
Z_{\h}(\WP^1_{(1,2)}(\Q), X) 	&	= \frac{12}{\pi^2} X^{4} + O\left( X^{2} \log X \right),
\end{split}
\]
by \cref{prop-hcount} and, for $\hst$, by a direct count: a point with $x_0 x_1 \neq 0$ has exactly two integral representatives $\pm(x_0, x_1)$ with no $d > 1$ satisfying $d \mid x_0$, $d^2 \mid x_1$, on which $\hst = \max(|x_0|, |x_1|^{1/2})$, so $Z_{\hst} = \tfrac12 \sum_{d \leq X} \mu(d) (2X/d)(2X^2/d^2) + O(X^2) = 2 X^3/\zeta(3) + O(X^2)$; the presentation $(1,2)$ is not well-formed, so this case is not covered by \cref{eq-deng} as stated. The discrepancy in the exponents is the counting-theoretic form of the defect identity of \cref{prop-1}: the points with $\h(\p) \leq X < \hst(\p)$ are exactly those with large defect.

On the image side, $\phi([x_0 : x_1]) = [x_0^2 : x_1]$ is surjective on $\Q$-points here, and in fact the same holds for every weighted projective line, since the exponents $q/q_0$ and $q/q_1$ are always coprime. For $n \geq 2$ the image can omit rational points: for $\w = (2,3,5)$, with $\phi = [x_0^{15} : x_1^{10} : x_2^{6}]$, the point $[1 : 2 : 1] \in \bP^2(\Q)$ has no rational preimage, since $x_0^{15} = \lambda$ and $x_2^{6} = \lambda$ force $v_2(\lambda) \equiv 0 \pmod{30}$, while $x_1^{10} = 2\lambda$ forces $10 \, v_2(x_1) = 1 + v_2(\lambda)$, an odd number. The precise solvability criterion, and the density of the image, are the subject of \cref{sec-5}; the counting exponents for $(2,3,5)$ are determined in \cref{subsec-235}.
\qed
\end{exa}

The bijectivity used in the proof of \cref{prop-hcount} holds for every presentation with pairwise coprime exponents and over every field, and gives the count over a number field.

\begin{lem}\label{lem-pc-bij}
Suppose the exponents $n_i = q/q_i$ are pairwise coprime. Then $\phi \colon \WP^n_{\w}(\bar k) \to \bP^n(\bar k)$ is a bijection, and for every field $K$ with $k \subseteq K \subseteq \bar{k}$ it restricts to a bijection $\WP^n_{\w}(K) \to \bP^n(K)$.
\end{lem}

\begin{proof}
The morphism $\phi$ preserves the support $S = \{ i : x_i \neq 0 \}$ of a point, since $x_i^{n_i} = 0$ if and only if $x_i = 0$, and if the $n_i$ are pairwise coprime then so are the $n_i$ with $i \in S$ for any $S$. We argue one support at a time.

We first prove that $\phi$ is injective on $\WP^n_{\w}(\bar{k})$, which gives the injectivity for every $K$ at once. Let $\p = [x_0 : \cdots : x_n]$ and $\p' = [x_0' : \cdots : x_n']$ have $\phi(\p) = \phi(\p')$. They have the same support $S$, and there is $c \in \bar{k}^\times$ with $x_i'^{\,n_i} = c \, x_i^{n_i}$ for every $i$. Fix $\lambda \in \bar{k}^\times$ with $\lambda^q = c$. For $i \in S$ the ratio $\zeta_i = x_i'/x_i$ satisfies $\zeta_i^{n_i} = c$, and $(\lambda^{q_i})^{n_i} = \lambda^q = c$ as well, so $\xi_i = \zeta_i \lambda^{-q_i}$ lies in $\mu_{n_i}(\bar{k})$. Consider the homomorphism
\[
\Psi \colon \mu_q(\bar{k}) \longrightarrow \prod_{i \in S} \mu_{n_i}(\bar{k}), \qquad \nu \longmapsto \left( \nu^{q_i} \right)_{i \in S},
\]
well defined because $(\nu^{q_i})^{n_i} = \nu^q = 1$. If $\zeta$ generates the cyclic group $\mu_q(\bar{k})$, then $\zeta^{q_i}$ has exact order $q/q_i = n_i$, so $\Psi(\zeta)$ has order $\lcm_{i \in S} n_i = \prod_{i \in S} n_i$; the target is a product of cyclic groups of pairwise coprime orders, hence cyclic of that same order, so $\Psi(\zeta)$ generates it and $\Psi$ is surjective. Choose $\nu \in \mu_q(\bar{k})$ with $\Psi(\nu) = (\xi_i)_{i \in S}$. Then $(\lambda\nu)^{q_i} x_i = \lambda^{q_i} \xi_i x_i = \zeta_i x_i = x_i'$ for $i \in S$, while both sides vanish for $i \notin S$, so $\p' = (\lambda\nu) \star \p = \p$.

For the surjectivity, let $y = [y_0 : \cdots : y_n] \in \bP^n(K)$ have support $S$, with coordinates $y_i \in K$. Since the $n_j$ with $j \in S$ are pairwise coprime, the Chinese remainder theorem provides, for each $j \in S$, an integer $c_j$ with $c_j \equiv -1 \pmod{n_j}$ and $c_j \equiv 0 \pmod{n_i}$ for every $i \in S \setminus \{j\}$. Put $\lambda = \prod_{j \in S} y_j^{c_j} \in K^\times$ and fix $i \in S$. In the monomial $\lambda y_i$ the exponent of $y_i$ is $c_i + 1$ and the exponent of $y_j$ is $c_j$ for $j \in S \setminus \{i\}$, and every one of them is divisible by $n_i$; hence $\lambda y_i = x_i^{n_i}$ for an explicit $x_i \in K^\times$. Setting $x_i = 0$ for $i \notin S$ gives $\p = [x_0 : \cdots : x_n] \in \WP^n_{\w}(K)$ with $\phi(\p) = [\lambda y_0 : \cdots : \lambda y_n] = y$.

\end{proof}

\begin{rem}\label{rem-chart-k}
Suppose the exponents $n_i$ are pairwise coprime and let $k$ be a number field of degree $m$. By \cref{lem-pc-bij} and \cref{eq-h-pullback}, $\phi$ is a height-preserving bijection $\WP^n_{\w}(k) \to \bP^n(k)$ with $\h(\p)^q = H(\phi(\p))$, so Schanuel's theorem gives
\[
Z_{\h}\big( \WP^n_{\w}(k), X \big) = Z_H\big( \bP^n(k), X^q \big) = S_k(n) \, X^{qm(n+1)} + O\big( X^{qm(n+1) - q} \log X \big)
\]
for every $n \geq 1$, where $Z_{\h}$ counts the points with $\h_k \leq X^m$; over $\Q$ and for $\w = (1, q, \dots, q)$ this is \cref{prop-hcount}. In particular the interior count $Z^{\circ}_{\h}$ has exponent $qm(n+1)$, the projective benchmark, and no sparsity occurs. \cref{cor:surjectivity} below shows that pairwise coprimality of the exponents is also necessary for $\phi$ to be surjective on $\Q$-points, and \cref{lem:wf-pairs} that it never holds for well-formed weights with $q \geq 2$; the case of pairwise coprime exponents is therefore exactly the case in which $\phi$ is the well-formalization isomorphism $\WP^n_{\w} \cong \bP^n$.
\end{rem}

\section{Kummer Obstructions and the Lifting Problem}\label{sec-5}

In this section we study the image of $\phi \colon \WP^n_{\w}(k) \to \bP^n(k)$ on rational points. We describe the fibers of $\phi$ as torsors under a finite group of roots of unity, make the obstruction completely explicit over $\Q$, compute the number of rational points in each fiber, show that whenever $\phi$ fails to be surjective on $\Q$-points, in particular for every well-formed $\w$ with $q \geq 2$, the image has density zero in $\bP^n(\Q)$, describe the obstruction over a number field, and compare the points admitting a $k$-representative with the $k$-points of the variety. The valuation-theoretic form of the criterion, the lifting set \cref{eq-monoid} below, is the combinatorial input for the counting theory of \cref{sec-main}. Throughout this section we write $n_i = q/q_i$ for the exponents of $\phi$; note that $\gcd(n_0, \dots, n_n) = q/\lcm(q_0, \dots, q_n) = 1$.

\subsection{The fibers of $\phi$ as torsors}

For $y = [y_0 : \dots : y_n] \in \bP^n(k)$ with all $y_i \neq 0$, a lift to $\WP^n_{\w}(k)$ is a solution of
\begin{equation}\label{eq-lift-system}
x_i^{n_i} = \lambda y_i \quad (0 \leq i \leq n), \qquad \lambda \in k^\times, \ x_i \in k^\times .
\end{equation}
By Kummer theory $H^1(k, \mu_{n}) \cong k^\times / (k^\times)^{n}$, the class of $\alpha$ being the obstruction to $\alpha$ being an $n$-th power. Thus $y$ lifts if and only if there is $\lambda \in k^\times$ with $\lambda y_i \in (k^\times)^{n_i}$ for every $i$; rescaling the coordinates of $y$ by $c \in k^\times$ translates the tuple $([y_i])_i \in \prod_i k^\times/(k^\times)^{n_i}$ by the diagonal image of $c$, so the class of $y$ in the quotient of $\prod_i k^\times/(k^\times)^{n_i}$ by the diagonal image of $k^\times$ depends only on the point, and $y$ lifts if and only if that class is trivial.

Geometrically, two preimages of $y$ differ by a tuple $(\zeta_0, \dots, \zeta_n)$ with $\zeta_0^{n_0} = \cdots = \zeta_n^{n_n}$, taken modulo the weighted scaling, so the fiber $\phi^{-1}(y)$ is a torsor under
\[
G = T / \iota(\bG_m), \qquad T = \Big\{ (\zeta_i)_i \in \prod_i \bG_m : \zeta_0^{n_0} = \cdots = \zeta_n^{n_n} \Big\}, \qquad \iota(\mu) = (\mu^{q_0}, \dots, \mu^{q_n}) .
\]
The map $(\zeta_i)_i \mapsto \zeta_0^{n_0}$ presents $T$ as an extension of $\bG_m$ by $\prod_i \mu_{n_i}$, so $\iota(\bG_m)$ is the identity component of $T$, $G = \pi_0(T) \cong \prod_i \mu_{n_i} / \iota(\mu_q)$ is finite, and since the kernel of $\iota$ on $\mu_q$ is $\mu_d$ with $d = \gcd(q_0, \dots, q_n)$,
\[
\# G = \frac{\prod_{i} n_i}{q/d} = \frac{q^n d}{\prod_i q_i} = d_\phi ,
\]
the degree of $\phi$ computed in \cref{lem:d_phi_definition}. The fibers without $k$-points are the nontrivial classes in $H^1(k, G)$; for descent obstructions of this type see \cite{skorobogatov2001torsors, poonen2006rational}.

On the stacky model $[(\A^{n+1} \setminus \{0\}) / \bG_m]$ with weights $\w$, the stabilizer of a point with support $S$ is $\mu_{g_S}$, $g_S = \gcd(q_i : i \in S)$. For well-formed $\w$ with $q \geq 2$ the stacky locus has codimension at least two, $\phi$ is a finite cover of degree $d_\phi > 1$, and the sparsity is carried by the torsors above rather than by stabilizers in codimension one. For the chart family $(1, q, \dots, q)$ every point of $\{x_0 = 0\}$ has stabilizer $\mu_q$ and every point off it has trivial stabilizer: the stack is the $q$-th root stack of $\bP^n$ along $\{y_0 = 0\}$, $\phi$ is its coarse space map, and no Kummer condition survives (\cref{cor:surjectivity}).

\subsection{The Lifting Criterion over \(\Q\)}

Over $\Q$ the Kummer classes are governed by valuations and signs, since 
\[
\Q^\times \cong \{\pm 1\} \times \bigoplus_p \Z,
\]
 and the obstruction becomes completely explicit. Its valuation-theoretic form is recorded by the \emph{lifting set}
\begin{equation}\label{eq-monoid}
M_{\w} = \left\{ t \in \Z_{\geq 0}^{n+1} \ : \ \min_i t_i = 0, \quad t_i \equiv t_j \!\!\!\pmod{\gcd(n_i, n_j)} \ \text{ for all } i, j \right\} .
\end{equation}
The normalization $\min_i t_i = 0$ corresponds to primitivity of an integral representative, and the congruences are the Kummer conditions of the next proposition.

\begin{prop}\label{prop:lifting-criterion}
Let $\w=(q_0,\dots,q_n)$, $n_i = q/q_i$, and $\phi$ as in \cref{Veronese}. A point 
\[
y=[y_0:\dots:y_n]\in \bP^n(\Q)
\]
 with all $y_i \neq 0$ admits a rational preimage under $\phi$ if and only if the following two conditions hold:
 
\begin{enumerate}[label=\roman*)]
    \item for every prime $p$ and all pairs $i,j$,
    \[    v_p(y_i) \equiv v_p(y_j) \pmod{\gcd(n_i, n_j)};    \]
    \item the coordinates $y_i$ for which $n_i$ is even all have the same sign.
\end{enumerate}
Equivalently, for a primitive integral representative, condition i) states that the valuation vector $\big( v_p(y_0), \dots, v_p(y_n) \big)$ lies in $M_{\w}$ for every prime $p$. For points with vanishing coordinates the same criterion applies to the sub-tuple of nonzero coordinates.
\end{prop}

\begin{proof}
A preimage exists if and only if there is $\lambda\in\Q^\times$ with $\lambda y_i \in (\Q^\times)^{n_i}$ for all $i$. An element $x\in\Q^\times$ is an $n_i$-th power if and only if $v_p(x)\equiv 0 \pmod{n_i}$ for every prime $p$, and $x>0$ when $n_i$ is even. The valuation conditions require, at each prime $p$, that 
\[
v_p(\lambda) \equiv -v_p(y_i) \pmod{n_i} \quad \text{ for all } \quad i;
\]
 by the Chinese remainder theorem such a value exists if and only if 
\[
v_p(y_i)\equiv v_p(y_j) \pmod{\gcd(n_i,n_j)}, \quad \text{  for all } \; i,j .
\]
 The choices at distinct primes are independent, and $v_p(\lambda) = 0$ is admissible at every prime where all $v_p(y_i) = 0$, so only finitely many primes are constrained and the prescribed valuations are realized by a rational number. The sign of $\lambda$ must satisfy $\operatorname{sign}(\lambda) = \operatorname{sign}(y_i)$ for every $i$ with $n_i$ even, which is possible if and only if those signs agree. 
 
 Conversely, any $\lambda$ with these valuations and sign realizes the lifting. For a primitive representative $\min_i v_p(y_i) = 0$ at every prime, which is the normalization $\min_i t_i = 0$ in \cref{eq-monoid}. Finally, $x_i^{n_i} = \lambda y_i$ with $\lambda \neq 0$ forces $x_i = 0$ if and only if $y_i = 0$, so a preimage has the same support as $y$ and the argument applies verbatim to the sub-tuple of nonzero coordinates.
\end{proof}

\begin{cor}\label{cor:surjectivity}
The morphism $\phi$ is surjective on $\Q$-points if and only if the exponents $n_0, \dots, n_n$ are pairwise coprime.
\end{cor}

\begin{proof}
If the $n_i$ are pairwise coprime, condition i) of \cref{prop:lifting-criterion} is vacuous; moreover at most one $n_i$ can be even, so condition ii) is vacuous as well, and the same holds on every coordinate stratum. 

Conversely, if $\gcd(n_i, n_j) > 1$ for some $i \neq j$, pick any prime $p$ and let $y$ have $y_i = p$ and all other coordinates $1$; then $v_p(y_i) = 1 \not\equiv 0 = v_p(y_j) \pmod{\gcd(n_i, n_j)}$, so $y$ has no rational preimage.
\end{proof}

For the chart family $\w = (1, q, \dots, q)$ of \cref{prop-hcount} the exponents are $(q, 1, \dots, 1)$, which are pairwise coprime; \cref{cor:surjectivity} thus explains structurally why no sparsity occurs there. Note that this family is not well-formed for $q \geq 2$. The next lemma shows that this is no accident: for $q \geq 2$, well-formedness is incompatible with pairwise coprimality.

\begin{lem}\label{lem:wf-pairs}
Let $\w$ be well-formed with $q \geq 2$. Then for every prime $\ell \mid q$ there exist indices $i \neq j$ with $\ell \mid \gcd(n_i, n_j)$. In particular the exponents $n_i$ are never pairwise coprime, $\phi$ is not surjective on $\Q$-points, and condition i) of \cref{prop:lifting-criterion} is nontrivial at every rational prime.
\end{lem}

\begin{proof}
Let $a = v_\ell(q) \geq 1$. Since $q_i \mid q$,
\[
v_\ell(n_i) = a - v_\ell(q_i) \geq 0 ,
\]
so $\ell \nmid n_i$ exactly when $v_\ell(q_i) = a$, and then $\ell \mid q_i$ because $a \geq 1$.

Well-formedness caps at $n-1$ the number of weights divisible by $\ell$: a set of $n$ or more indices $j$ with $\ell \mid q_j$ contains $\{0, \dots, n\} \setminus \{i\}$ for some $i$, forcing $\ell \mid \gcd(q_j : j \neq i) = 1$. Hence at most $n-1$ indices satisfy $\ell \nmid n_i$, so at least two of the $n+1$ indices satisfy $\ell \mid n_i$, and any such pair $i \neq j$ has $\ell \mid \gcd(n_i, n_j)$.

For the final statements, $q \geq 2$ supplies a prime $\ell \mid q$, and the resulting pair has $\gcd(n_i, n_j) \geq \ell > 1$: the exponents are not pairwise coprime, $\phi$ is not surjective on $\Q$-points by \cref{cor:surjectivity}, and condition i) of \cref{prop:lifting-criterion} constrains $v_p(y_i)$ and $v_p(y_j)$ modulo $\gcd(n_i, n_j) > 1$ at every rational prime $p$.
\end{proof}

\subsection{Fiber Counts and the Density of the Image}

We next count the rational points in a fiber, which is where the roots of unity of $\Q$ enter: the group $\mu_{n_i}(\Q)$ is $\{\pm 1\}$ or trivial according to the parity of $n_i$.

\begin{lem}\label{lem:fiber-count}
Assume $\gcd(q_0, \dots, q_n) = 1$ and let $y \in \bP^n(\Q)$ with all $y_i \neq 0$ admit a rational preimage. Let $\nu = \#\{ i : n_i \text{ even} \}$. Then the number of rational points in $\phi^{-1}(y)$ is
\[
\# \phi^{-1}(y)( \Q) =
\begin{cases}
2^{\nu - 1}, & q \text{ even}, \\
1, & q \text{ odd}.
\end{cases}
\]
\end{lem}

\begin{proof}
The set of admissible $\lambda$ is a coset of $\bigcap_i (\Q^\times)^{n_i}$, since the ratio of two admissible scalars lies in every $(\Q^\times)^{n_i}$, and
\[
\bigcap_i (\Q^\times)^{n_i} = (\Q^\times)^{\lcm(n_0, \dots, n_n)} = (\Q^\times)^{q} ,
\]
the first equality because a rational number is an $n$-th power exactly when all its valuations are divisible by $n$ and it is positive for $n$ even, while $\lcm(n_0, \dots, n_n)$ is even precisely when some $n_i$ is, and the second because $\lcm_i (q/q_i) = q/\gcd(q_0, \dots, q_n) = q$. 

Replacing a preimage $[x_0 : \dots : x_n]$ by $\mu \star [x_0 : \dots : x_n]$ scales $\lambda$ by $\mu^q$, and $\{ \mu^q : \mu \in \Q^\times \} = (\Q^\times)^q$, so the action of $\Q^\times$ is transitive on the admissible $\lambda$ and we may fix one $\lambda_0$. For fixed $\lambda_0$, each equation $x_i^{n_i} = \lambda_0 y_i$ has one rational solution if $n_i$ is odd and two if $n_i$ is even, giving $2^{\nu}$ solution tuples. The residual identifications are by $\mu$ with $\mu^q = 1$, i.e., $\mu = \pm 1$ if $q$ is even and $\mu = 1$ if $q$ is odd. If $q$ is even, $\mu = -1$ acts on tuples by $x_i \mapsto (-1)^{q_i} x_i$; a fixed tuple with all coordinates nonzero would force every $q_i$ even, contradicting $\gcd(q_0, \dots, q_n) = 1$, so the action is free and the $2^\nu$ tuples form $2^{\nu - 1}$ points. If $q$ is odd, every $n_i$ is odd, so $\nu = 0$ and the single tuple gives a single point.

\end{proof}

Note that when $q$ is even, coprimality of the weights forces $\nu \geq 1$, so the count $2^{\nu - 1}$ is a positive integer. For the chart family $(1, q, \dots, q)$ one has $\nu = 1$ if $q$ is even and $\nu = 0$ otherwise, so every fiber over the image contains exactly one rational point, in agreement with the bijection of \cref{lem-pc-bij}. Since fiber multiplicities are bounded, the counting function $Z_{\h}$ is controlled, up to bounded factors, by the number of liftable points in $\bP^n(\Q)$, and \cref{lem:wf-pairs} shows that for well-formed weights with $q \geq 2$ the liftability conditions accumulate at every prime. This forces the image to be sparse in the strongest sense.

\begin{prop}\label{prop:density-zero}
Let $\w = (q_0, \dots, q_n)$ and let $\phi$ be as in \cref{Veronese}. Exactly one of the following holds:
\begin{enumerate}[label=\roman*)]
    \item the exponents $n_0, \dots, n_n$ are pairwise coprime, and $\phi$ is surjective on $\Q$-points;

    \item the set of $y \in \bP^n(\Q)$ admitting a rational preimage under $\phi$ has density zero with respect to the Weil height, and
    \[
    Z_{\h}(\WP^n_{\w}(\Q), X) = o\left( X^{q(n+1)} \right) .
    \]
\end{enumerate}
By \cref{lem:wf-pairs}, case ii) occurs for every well-formed $\w$ with $q \geq 2$.
\end{prop}

\begin{proof}
If the exponents are pairwise coprime, then i) holds by \cref{cor:surjectivity} and the image is all of $\bP^n(\Q)$, of density one, so ii) fails; the two cases are exclusive, and it remains to prove ii) when $\gcd(n_i, n_j) > 1$ for some $i \neq j$.

Let $S = \{ i : \gcd(n_i, n_j) > 1 \text{ for some } j \neq i \}$, so that $|S| \geq 2$, both indices of such a pair lying in $S$. For a prime $p$ and $i \in S$, consider the event $A_p(i)$ on primitive integral tuples $(y_0, \dots, y_n)$: $p \parallel y_i$ and $p \nmid y_j$ for all $j \neq i$. If $A_p(i)$ holds, then every coordinate is nonzero and $v_p(y_i) = 1 \not\equiv 0 = v_p(y_j) \pmod{\gcd(n_i, n_j)}$ for a witness $j$, so condition i) of \cref{prop:lifting-criterion} fails and $y$ is not liftable. Each $A_p(i)$ is a congruence condition modulo $p^2$, of density
\[
\frac{(p-1) p^{-2} \cdot \left( (p-1) p^{-1} \right)^{n}}{1 - p^{-(n+1)}} = \frac{1}{p} + O\left( \frac{1}{p^2} \right)
\]
among primitive tuples, and for fixed $p$ the events $A_p(i)$, $i \in S$, are disjoint. For any finite set $P$ of primes $p > |S|$, counting lattice points in the finitely many admissible residue classes modulo $\prod_{p \in P} p^2$ shows that the density of primitive tuples avoiding all $A_p(i)$ with $p \in P$, $i \in S$, is
\[
\prod_{p \in P} \left( 1 - \frac{|S|}{p} + O\left( \frac{1}{p^2} \right) \right),
\]
which tends to $0$ as $P$ exhausts the primes, by Mertens' theorem. Hence the liftable points have density zero.

For the counting bound, \cref{eq-h-pullback} sends a point $\p \in \WP^n_{\w}(\Q)$ with $\h(\p) \leq X$ to the liftable point $y = \phi(\p)$ of $\bP^n(\Q)$ with $H(y) \leq X^q$, and the fibers of $\phi$ have at most $\prod_i n_i$ points, since over $\bar{\Q}$ the scalar $\lambda$ may be normalized to $1$ and a preimage is then determined by a choice of $n_i$-th root of unity at each nonzero coordinate. As the $y \in \bP^n(\Q)$ with $H(y) \leq X^q$ number $O(X^{q(n+1)})$ by \cite{schanuel1979}, and the liftable ones among them have density zero, the bound follows.
\end{proof}


\subsection{Beyond \(\Q\)}

We record the shape of the obstruction over a number field, which motivates the conjecture of \cref{sec-7}. Let $L$ be a number field and let $y \in \bP^n(L)$ have all coordinates nonzero. For a choice of homogeneous coordinates $(y_0, \dots, y_n)$ of $y$, the \emph{Kummer class} of $y$ is the class of $\big( [y_i] \big)_i$ in the quotient of
\[
\prod_{i=0}^{n} L^\times/(L^\times)^{n_i} \cong \prod_{i=0}^{n} H^1(L, \mu_{n_i})
\]
by the image of the diagonal map $c \mapsto \big( [c] \big)_i$, the isomorphism being the Kummer isomorphism of \cref{sec-5}.

\begin{prop}\label{prop:number-field}
Let $\w = (q_0, \dots, q_n)$, $n_i = q/q_i$, and $\phi$ as in \cref{Veronese}, let $L$ be a number field, and let $y = [y_0 : \dots : y_n] \in \bP^n(L)$ have all coordinates nonzero. Then:
\begin{enumerate}[label=\roman*)]
    \item the Kummer class of $y$ does not depend on the choice of homogeneous coordinates;

    \item $y$ lies in $\phi\big( \WP^n_{\w}(L) \big)$ if and only if the Kummer class of $y$ is trivial;

    \item if $y$ lies in $\phi\big( \WP^n_{\w}(L) \big)$, then for every finite place $v$ of $L$ and all pairs $i, j$,
    \[
    v(y_i) \equiv v(y_j) \pmod{\gcd(n_i, n_j)} .
    \]
\end{enumerate}
\end{prop}

\begin{proof}
i) Two systems of homogeneous coordinates for $y$ differ by a factor $c \in L^\times$, and $\big( [c y_i] \big)_i$ differs from $\big( [y_i] \big)_i$ by the diagonal image of $c$, so the two classes agree in the quotient.

ii) A point $\p = [x_0 : \dots : x_n] \in \WP^n_{\w}(L)$ satisfies $\phi(\p) = y$ if and only if $x_i^{n_i} = \lambda y_i$ for all $i$ and some $\lambda \in L^\times$. Given such $\p$ and $\lambda$, every $\lambda y_i$ is an $n_i$-th power, so $[y_i] = [\lambda]^{-1}$ in $L^\times/(L^\times)^{n_i}$ and $\big( [y_i] \big)_i$ is the diagonal image of $\lambda^{-1}$. 

Conversely, if the Kummer class is trivial, then $\big( [y_i] \big)_i$ is the diagonal image of some $c \in L^\times$, that is, $c^{-1} y_i \in (L^\times)^{n_i}$ for all $i$; setting $\lambda = c^{-1}$ and choosing $x_i \in L^\times$ with $x_i^{n_i} = \lambda y_i$, which is possible since $\lambda y_i \neq 0$, gives 
\[
\phi([x_0 : \dots : x_n]) = [\lambda y_0 : \dots : \lambda y_n] = y.
\]

iii) Apply $v$ to $x_i^{n_i} = \lambda y_i$: then 
\[
v(\lambda) + v(y_i) = n_i v(x_i),
\]
 so $v(\lambda) \equiv -v(y_i) \pmod{n_i}$ for every $i$. Reducing the congruences for $i$ and $j$ modulo $\gcd(n_i, n_j)$ gives $-v(y_i) \equiv v(\lambda) \equiv -v(y_j)$.
\end{proof}

A point of $\bP^n(L)$ with support $S \subsetneq \{0, \dots, n\}$ has all its preimages supported on $S$, by the proof of \cref{prop:lifting-criterion}, so \cref{prop:number-field} applies to it with every product and every pair of indices restricted to $S$ and the exponents $n_i$ unchanged.

\begin{exa}\label{exa-4}
The converse of \cref{prop:number-field}~iii) fails over a number field: the congruences are necessary but not sufficient. Let $\w = (2,3,5)$, so that 
\[
(n_0, n_1, n_2) = (15, 10, 6).
\]
Let $L = \Q(i)$ and  $y = [1 : i : 1] \in \bP^2(L)$. Each coordinate of $y$ is a unit at every finite place of $L$, so all valuations vanish and the congruences of \cref{prop:number-field}~iii) hold at every finite place. Yet $y$ has no preimage in $\WP^2_{(2,3,5)}(L)$. A lift requires 
\[
\lambda \in (L^\times)^{15} \cap (L^\times)^{6} = (L^\times)^{30},
\]
 the equality holding in any abelian group because $\gcd(30/15, 30/6) = 1$, and then
\[
i = (\lambda i) \, \lambda^{-1} \in (L^\times)^{10} \subseteq (L^\times)^{2} .
\]
But $i = z^2$ forces $z^8 = 1$ and $z^4 = -1$, so $z$ has order eight, while $\mu(L) = \{\pm 1, \pm i\}$ has order four. The obstruction is the unit $i$, invisible to the valuations; over $\Q$ the same weights are obstructed instead by a valuation, as in \cref{exa-3}.
\qed
\end{exa}

\begin{rem}\label{rem-selmer}
The failure has two sources, both of which degenerate over $\Q$. Write $N = \lcm(n_0, \dots, n_n)$. At each finite place $v$ the admissible values of $v(\lambda)$ form a coset of $N\Z$, containing $0$ at all but finitely many $v$; realizing a choice of such values by a single $\lambda \in L^\times$ requires the ideal 
\[
\prod_v \fP_v^{v(\lambda)}
\]
 to be principal, and moving $v(\lambda)$ within its coset multiplies that ideal by an $N$-th power, so the obstruction is a class in $\mathrm{Cl}(L)/\mathrm{Cl}(L)^N$. 
 
 Once $\lambda$ realizes admissible valuations, each $\lambda y_i$ has all valuations divisible by $n_i$ but need not be an $n_i$-th power, the discrepancy being the Selmer-type group
\[
\mathrm{Sel}_n(L) = \ker \left( L^\times/(L^\times)^n \longrightarrow \bigoplus_{v \in M_L^0} \Z/n \right) ,
\]
which sits in the exact sequence
\[
1 \longrightarrow \cO_L^\times/(\cO_L^\times)^n \longrightarrow \mathrm{Sel}_n(L) \longrightarrow \mathrm{Cl}(L)[n] \longrightarrow 1 ,
\]
obtained by sending $x$ with $(x) = \fa^n$ to $[\fa]$, the kernel being the image of $\cO_L^\times$ because 
\[
\cO_L^\times \cap (L^\times)^n = (\cO_L^\times)^n.
\] 
Over $\Q$ the class group is trivial and 
\[
\mathrm{Sel}_n(\Q) = \{\pm 1\}(\Q^\times)^n/(\Q^\times)^n
\]
 is $\Z/2$ for $n$ even and trivial for $n$ odd, which is precisely condition ii) of \cref{prop:lifting-criterion}; that is why valuations and signs give a complete criterion over $\Q$, and \cref{exa-4} is the nontrivial class of the unit $i$ in $\mathrm{Sel}_{10}(\Q(i))$. This is the Selmer-type contribution of \cref{conj-1}.
\end{rem}


\subsection{Representable points and the field of definition}\label{subsec-fod}

For a field $K$ with $k \subseteq K \subseteq \bar k$ consider
\begin{align*}
\Ad_K &= \WP^n_{\w}(K) = \{ \p \in \WP^n_{\w}(\bar k) : \p \text{ has a representative } \x \in K^{n+1} \}, \\
\mathcal{B}_K &= \{ \p \in \WP^n_{\w}(\bar k) : \p^{\sigma} = \p \text{ for all } \sigma \in \Gal(\bar k / K) \}, \\
\mathcal{C}_K &= \{ \p \in \WP^n_{\w}(\bar k) : \phi(\p) \in \bP^n(K) \} .
\end{align*}
Here $\mathcal{B}_K$ is the set of $K$-points of the variety $\WP^n_{\w}$ and $\mathcal{C}_K = \phi^{-1}(\bP^n(K))$, and $\Ad_K \subseteq \mathcal{B}_K \subseteq \mathcal{C}_K$. By \cref{cor:surjectivity}, $\mathcal{C}_\Q \neq \Ad_\Q$ as soon as the exponents fail to be pairwise coprime; for instance $[1 : 2^{1/10} : 1] \in \WP^2_{(2,3,5)}(\bar\Q)$ lies in $\mathcal{C}_\Q$ but not in $\Ad_\Q$. On the interior the first inclusion is an equality.

\begin{prop}\label{prop:fod}
Let $\gcd(q_0, \dots, q_n) = 1$ and let $\p \in \WP^n_{\w}(\bar k)$ have all coordinates nonzero. Then for every $K$ with $k \subseteq K \subseteq \bar k$,
\[
\p \in \Ad_K \iff \p \in \mathcal{B}_K .
\]
Consequently the minimal field over which $\p$ admits a representative equals the fixed field $k(\p)$ of the stabilizer of $\p$ in $\Gal(\bar k/k)$, and this field is a finite extension of $k(\phi(\p))$ of degree at most $d_\phi$.
\end{prop}

\begin{proof}
Only $\mathcal{B}_K \subseteq \Ad_K$ needs proof. Let $\x = (x_0, \dots, x_n) \in (\bar k^\times)^{n+1}$ represent $\p$, and suppose $\p^\sigma = \p$ for all $\sigma \in \Gal(\bar k / K)$. For each $\sigma$ there is $\lambda_\sigma \in \bar k^\times$ with $\sigma(\x) = \lambda_\sigma \star \x$, that is, $\sigma(x_i) = \lambda_\sigma^{q_i} x_i$ for all $i$. The scalar $\lambda_\sigma$ is unique: if $\lambda \star \x = \x$ with all $x_i \neq 0$, then $\lambda^{q_i} = 1$ for all $i$, hence $\lambda^{\gcd(q_0, \dots, q_n)} = \lambda = 1$. Uniqueness gives the cocycle relation $\lambda_{\sigma\tau} = \lambda_\sigma \, \sigma(\lambda_\tau)$: indeed $\sigma\tau(\x) = \sigma(\lambda_\tau \star \x) = \sigma(\lambda_\tau) \star \sigma(\x) = (\sigma(\lambda_\tau) \lambda_\sigma) \star \x$. Moreover $\lambda_\sigma$ lies in a fixed finite extension of $K$, since $\lambda_\sigma^{q_i} = \sigma(x_i)/x_i \in k(\x)$ and $\lambda_\sigma = \prod_i (\lambda_\sigma^{q_i})^{a_i}$ for integers $a_i$ with $\sum a_i q_i = 1$; so $\sigma \mapsto \lambda_\sigma$ is a continuous $1$-cocycle of $\Gal(\bar k/K)$ with values in $\bar k^\times$. By Hilbert's Theorem 90 there is $\mu \in \bar k^\times$ with $\lambda_\sigma = \sigma(\mu)/\mu$ for all $\sigma$. Put $\x' = \mu^{-1} \star \x$, another representative of $\p$. Then
\[
\sigma(x'_i) = \sigma(\mu)^{-q_i} \sigma(x_i) = \sigma(\mu)^{-q_i} \lambda_\sigma^{q_i} x_i = \mu^{-q_i} x_i = x'_i ,
\]
so $\x' \in K^{n+1}$ and $\p \in \Ad_K$.

For the last statement, the stabilizer of $\p$ is closed of finite index, since $\Gal(\bar k/k)$ permutes the finitely many preimages of $\phi(\p)$ (\cref{lem:d_phi_definition}) and $\p$ is one of them; its fixed field $k(\p)$ contains $k(\phi(\p))$, because $\phi(\p^\sigma) = \phi(\p)^\sigma$, and $[k(\p) : k(\phi(\p))]$ is the size of the orbit of $\p$ under $\Gal(\bar k / k(\phi(\p)))$, which acts on the fiber, a set of at most $d_\phi$ points. Applying the first part with $K = k(\p)$ shows that $\p$ has a representative over $k(\p)$, and no smaller field can carry a representative since $\Ad_K \subseteq \mathcal{B}_K$.
\end{proof}

\begin{rem}\label{rem-fod}
The hypothesis that all coordinates be nonzero says that the stabilizer of $\p$ on the stack is trivial. On a coordinate stratum with $g_S = \gcd(q_i : i \in S) > 1$ the same argument applies after replacing $\w_S$ by $\w_S / g_S$, which does not change the point set (\cref{lem:stratum-height}); so on every stratum the points with a $\Q$-representative are the $\Q$-points of the variety.

\cref{prop:fod} also fixes the meaning of points of fixed degree: $\WP^n_{\w}(k;e) = \{\p : [k(\p):k] = e\}$ with $k(\p)$ the Galois-theoretic field of definition, partitioned by the minimal field of representation. With $k(\phi(\p))$ in place of $k(\p)$ one would instead obtain $\phi^{-1}(\bP^n(k;e))$, whose counting function is $d_\phi$ times that of $\bP^n(k;e)$ at the threshold $X^q$; the Kummer obstruction is the difference between $k(\p)$ and $k(\phi(\p))$. See \cref{sec-7}.
\end{rem}

\section{A Schanuel Theorem for the Weighted Height}\label{sec-main}

In this section we state and prove the main theorem of the paper: a Schanuel-type asymptotic for the weighted height on every weighted projective space over $\Q$ with coprime weights, together with the bounds $Q \leq q\,a(\w) \leq q(n+1)$ on its exponent, the stratification of the full counting function, and the sparsity exponent of \cref{cor:sparsity}. The analytic engine is the multivariable Tauberian theorem of de la Bret\`eche \cite{delaBreteche2001}. Throughout we write $n_i = q/q_i$ for the exponents of $\phi$, so that $\gcd(n_0, \dots, n_n) = 1$, and we assume $n \geq 1$.

\subsection{The lifting set, the linear program, and the statement}\label{subsec-main-statement}

Recall the lifting set $M_{\w}$ of \cref{eq-monoid}: by \cref{prop:lifting-criterion}, a point of $\bP^n(\Q)$ with a primitive integral representative with nonzero coordinates lifts along $\phi$ if and only if its valuation vector at every prime lies in $M_{\w}$, together with a sign condition. The set $M_{\w}$ is not closed under addition, since the sum of two vectors with vanishing minimum may have positive minimum; what enters the counting is the set
\[
T_{\w} = \left\{ t \in M_{\w} \setminus \{0\} \ : \ t \text{ is minimal in } M_{\w} \setminus \{0\} \text{ for the coordinatewise partial order} \right\},
\]
which is finite by Dickson's lemma, and every nonzero element of $M_{\w}$ dominates an element of $T_{\w}$. Consider the linear program
\begin{equation}\label{eq-LP}
\max \ \sum_{g \in T_{\w}} c_g \qquad \text{subject to} \qquad c_g \geq 0, \qquad \sum_{g \in T_{\w}} c_g \, g_i \leq 1 \quad (0 \leq i \leq n),
\end{equation}
and let $a(\w)$ be its value and $\beta(\w)$ the dimension of its optimal face. The feasible region is a polytope, since $g \neq 0$ forces $c_g \leq 1$, so both are well defined.

When the exponents $n_i$ are pairwise coprime, all congruences in \cref{eq-monoid} are vacuous, $T_{\w} = \{ e_0, \dots, e_n \}$, the $i$-th constraint of \cref{eq-LP} reads $c_{e_i} \leq 1$, and the optimum is the single point $c = \mathbf{1}$, so $a(\w) = n+1$ and $\beta(\w) = 0$. The converse holds as well. Summing the $n+1$ constraints gives
\begin{equation}\label{eq-a-bound}
\sum_{g \in T_{\w}} c_g \left( \sum_{i=0}^n g_i \right) \leq n+1 ,
\end{equation}
and $\sum_i g_i \geq 1$ for every $g \in T_{\w}$, so $a(\w) \leq n+1$ for all weights; if equality holds, then every optimal $c$ satisfies $c_g \big( \sum_i g_i - 1 \big) = 0$, hence is supported on $T_{\w} \cap \{ e_0, \dots, e_n \}$, and $n+1 = \sum_g c_g \leq \# \{ i : e_i \in T_{\w} \}$ forces $e_i \in M_{\w}$ for every $i$, which by \cref{eq-monoid} says $\gcd(n_i, n_j) = 1$ for all $i \neq j$. The invariants are therefore nontrivial exactly when Kummer conditions are present.

The linear program is also bounded from below, by the exponent of the tautological height. For each $i$ put $\ell_i = \lcm_{j \neq i} \gcd(n_i, n_j)$, so that $\ell_i \mid n_i$; by \cref{lem:L1}~iv) below, $\ell_i e_i$ is the unique multiple of $e_i$ in $T_{\w}$, and $\sigma_0 = (1/\ell_0, \dots, 1/\ell_n)$ is the point forced by the axis constraints of the dual program.

\begin{prop}\label{prop:a-bounds}
Let $n \geq 1$ and $\gcd(q_0, \dots, q_n) = 1$. Then
\[
Q \ \leq\ q \, a(\w) \ \leq\ q (n+1) .
\]
The upper bound is an equality if and only if the exponents $n_i$ are pairwise coprime. The lower bound is an equality if and only if
\begin{equation}\label{eq-Q-criterion}
\sum_{i=0}^n t_i \, q_i \ \geq\ q \qquad \text{for every nonzero } t \in M_{\w},
\end{equation}
in which case $\sigma^* = \w / q$ is the unique dual optimum. More precisely, $a(\w) \geq \sum_i 1/\ell_i \geq \sum_i 1/n_i = Q/q$.
\end{prop}

\begin{proof}
The upper bound and its case of equality were proved above. For the lower bound, let $\sigma^*$ be any optimal solution of the dual program
\[
\min \ \sum_i \sigma_i \qquad \text{subject to} \qquad \sigma \geq 0, \quad \langle g, \sigma \rangle \geq 1 \ (g \in T_{\w}) ,
\]
whose value is $a(\w)$ by strong duality (the primal is feasible and bounded). The constraint attached to $\ell_i e_i \in T_{\w}$ reads $\ell_i \sigma_i^* \geq 1$, so
\[
a(\w) = \sum_i \sigma_i^* \geq \sum_i \frac{1}{\ell_i} \geq \sum_i \frac{1}{n_i} = \sum_i \frac{q_i}{q} = \frac{Q}{q} .
\]
If $\sigma^*$ attains $Q/q$ then $\sigma_i^* = 1/n_i = q_i/q$ for every $i$, so $\sigma^* = \w/q$ is dual feasible, which is \cref{eq-Q-criterion} restricted to $T_{\w}$, hence to all of $M_{\w} \setminus \{0\}$ since every nonzero element of $M_{\w}$ dominates an element of $T_{\w}$ and the form $\langle \cdot, \w \rangle$ is monotone. Conversely, if \cref{eq-Q-criterion} holds then $\w/q$ is dual feasible with objective value $Q/q$, and any other dual feasible $\sigma$ satisfies $\sigma_i \geq 1/\ell_i \geq q_i/q$, so $\w / q$ is the unique optimum.
\end{proof}

\begin{rem}\label{rem:a-bounds}
For well-formed $\w$ the inequality $q\,a(\w) \geq Q$ also follows from $\h \leq \hst$, which gives $\{\hst \leq X\} \subseteq \{\h \leq X\}$, together with \cref{eq-deng} and \cref{thm:main}. Both equality and strict inequality occur, and equality is compatible with $\beta(\w) > 0$. For $(1,2,3,5)$ one checks \cref{eq-Q-criterion} directly (\cref{prop:1235}) and $q\,a(\w) = Q = 11$. For $(2,3,5)$ the vector $(5,0,2) \in M_{\w}$ has $\langle t, \w \rangle = 20 < 30$ and $q\,a(\w) = 12 > 10 = Q$ (\cref{subsec-235}). For $(1,2,3)$ the criterion holds, $q\,a(\w) = Q = 6$, but the optimal face has dimension $\beta(\w) = 3$, so the two heights count at the same exponent and $\h$ carries an additional $(\log X)^3$ (\cref{subsec-123}); the pairwise coprime case is the only one with $\beta(\w) = 0$ and $q\,a(\w) = q(n+1)$.
\end{rem}

The counting function of $\WP^n_{\w}(\Q)$ decomposes along the coordinate strata, which are themselves weighted projective spaces, and the weighted height is compatible with this stratification.

\begin{lem}\label{lem:stratum-height}
Let $\emptyset \neq S \subseteq \{0, \dots, n\}$, $\w_S = (q_i)_{i \in S}$, and $q_S = \lcm_{i \in S} q_i$. The locus $\{ x_j = 0 : j \notin S \}$ of $\WP^n_{\w}$ is $\WP^{|S|-1}_{\w_S}$, and on it the weighted height of $\WP^n_{\w}$ restricts to the weighted height of $\WP^{|S|-1}_{\w_S}$:
\[
\h_{\w}(\p) = \h_{\w_S}(\p) \qquad \text{for all } \p \in \WP^{|S|-1}_{\w_S}(\bar{\Q}) .
\]
\end{lem}

\begin{proof}
The locus consists of the classes of tuples $\x$ with $x_j = 0$ for $j \notin S$. The vanishing coordinates impose no condition on $\lambda$, so two such tuples satisfy $\x' = \lambda \star \x$ if and only if their restrictions to $S$ satisfy the same relation for the weights $\w_S$; restriction is therefore a bijection of the locus onto $\WP^{|S|-1}_{\w_S}$, and it carries tuples with entries in a field $k$ to tuples with entries in $k$ and back.

Write $d_S = q / q_S \in \Z_{\geq 1}$. For $\p$ supported on $S$, the $i$-th coordinate of $\phi_{\w}(\p)$ is $x_i^{q/q_i} = \big( x_i^{q_S/q_i} \big)^{d_S}$, so $\phi_{\w}(\p)$ is the image of $\phi_{\w_S}(\p)$ under the coordinatewise $d_S$-th power map on $\bP^{S}$, and $H\big( \phi_{\w}(\p) \big) = H\big( \phi_{\w_S}(\p) \big)^{d_S}$, the local factors at every place being raised to the power $d_S$. Hence, by \cref{eq-h-pullback} applied to $\w$ and to $\w_S$,
\[
\h_{\w}(\p) = H\big( \phi_{\w}(\p) \big)^{1/q} = H\big( \phi_{\w_S}(\p) \big)^{d_S / q} = H\big( \phi_{\w_S}(\p) \big)^{1/q_S} = \h_{\w_S}(\p) . \qedhere
\]
\end{proof}

Writing $Z^{\circ}_{\h}(\WP^n_{\w}(\Q), X)$ for the number of points of weighted height at most $X$ with \emph{all coordinates nonzero}, the strata partition the space by support, and \cref{lem:stratum-height} gives the exact identity
\begin{equation}\label{eq-stratification}
Z_{\h}\big( \WP^n_{\w}(\Q), X \big) \;=\; \sum_{\emptyset \neq S \subseteq \{0, \dots, n\}} Z^{\circ}_{\h}\big( \WP^{|S|-1}_{\w_S}(\Q), X \big) .
\end{equation}
For $|S| = 1$ the stratum is the single point $[0 : \cdots : 1 : \cdots : 0]$, of weighted height $1$, and $M_{\w_S} = \{0\}$, $T_{\w_S} = \emptyset$, $a(\w_S) = \beta(\w_S) = 0$; the summand equals $1$ for every $X \geq 1$. The main theorem determines the summands with $|S| \geq 2$.


\begin{thm} \label{thm:main}
Let $n \geq 1$, let $\w = (q_0, \dots, q_n)$ with $\gcd(q_0, \dots, q_n) = 1$, and let $a(\w)$ and $\beta(\w)$ be the value and the dimension of the optimal face of the linear program \cref{eq-LP}. Then there exist $\theta > 0$ and a polynomial $P_{\w} \in \R[x]$ of exact degree $\beta(\w)$ with positive leading coefficient $c_{\w}$ such that, as $X \to \infty$,
\[
Z^{\circ}_{\h}\big(\WP^n_{\w}(\Q), X\big) = X^{q \, a(\w)} \, P_{\w}(\log X) + O\big( X^{q \, a(\w) - \theta} \big) ;
\]
in particular $Z^{\circ}_{\h}\big(\WP^n_{\w}(\Q), X\big) \sim c_{\w} \, X^{q \, a(\w)} (\log X)^{\beta(\w)}$.
\end{thm}

Well-formedness is not required, and must not be: the apparatus below depends only on the exponents $n_i$, and every coprime weight vector arises as a stratum $\w_S$ in \cref{eq-stratification}. For pairwise coprime exponents $M_{\w}$ imposes only primitivity, and \cref{thm:diagonal} below gives $Z^{\circ}_{\h} \sim c_{\w} X^{q(n+1)}$ with $c_{\w} = 2^n / \zeta(n+1)$, recovering the constant of \cref{prop-hcount}. For a stratum with $d = \gcd_{i \in S} q_i > 1$ the identity $\h_{\w_S} = \h_{\w_S / d}^{1/d}$ gives $Z^{\circ}_{\h}(\WP^{|S|-1}_{\w_S}(\Q), X) = Z^{\circ}_{\h}(\WP^{|S|-1}_{\w_S / d}(\Q), X^{d})$ and reduces the count to the coprime case, the two presentations sharing the exponents $q_S / q_i$ and hence $a$ and $\beta$.

\begin{cor}\label{cor:full-count}
Let $n \geq 1$ and $\w = (q_0, \dots, q_n)$ with $\gcd(q_0, \dots, q_n) = 1$. For $S \subseteq \{0, \dots, n\}$ with $|S| \geq 2$ set $d = \gcd_{i \in S} q_i$ and
\[
\alpha_S = q_S \, a(\w_S), \qquad \beta_S = \beta(\w_S), \qquad c_{\w_S} = d^{\beta_S} \, c_{\w_S / d} ,
\]
the invariants $a(\w_S)$ and $\beta(\w_S)$ computed from the exponents $q_S / q_i$, $i \in S$, of the stratum and $c_{\w_S / d}$ the constant of \cref{thm:main}, and let $(\alpha, \beta)$ be the lexicographic maximum of the pairs $(\alpha_S, \beta_S)$. Then
\[
Z_{\h}\big( \WP^n_{\w}(\Q), X \big) \sim \Big( \sum_{S \,:\, (\alpha_S, \beta_S) = (\alpha, \beta)} c_{\w_S} \Big) \, X^{\alpha} (\log X)^{\beta} .
\]
A proper stratum with $(\alpha_S, \beta_S)$ lexicographically larger than the pair of the full support is an accumulating subvariety for the weighted height. Since $a(\w_S) \leq |S|$ by \cref{eq-a-bound}, one has $\alpha_S \leq q_S |S| \leq q n$ for proper strata, so $Z_{\h} = o\big( X^{q(n+1)} \big)$ whenever $a(\w) < n+1$, in accordance with \cref{prop:density-zero}.
\end{cor}

\begin{proof}
The decomposition \cref{eq-stratification} is a finite disjoint sum, and its $n+1$ strata with $|S| = 1$ contribute $1$ each. For $|S| \geq 2$ the reduction to $\w_S / d$ noted above and \cref{thm:main} give
\[
\begin{split}
Z^{\circ}_{\h}\big( \WP^{|S|-1}_{\w_S}(\Q), X \big)	&	\sim c_{\w_S / d} \, \big( X^{d} \big)^{(q_S / d) \, a(\w_S)} \big( d \log X \big)^{\beta_S} \\
										&	= c_{\w_S} \, X^{\alpha_S} (\log X)^{\beta_S} .
\end{split}
\]
Here $T_{\w_S} \neq \emptyset$, so $a(\w_S) > 0$ and $\alpha \geq \alpha_S > 0$; the singleton strata are therefore absorbed. Summands with lexicographically smaller pairs are $o\big( X^{\alpha} (\log X)^{\beta} \big)$, and the maximal summands add.
\end{proof}

Accumulation by strata occurs: for $\w = (2,3,5)$ the boundary line $\{x_0 = 0\} = \WP^1_{(3,5)}$ has pairwise coprime exponents $(5,3)$, so $(\alpha_S, \beta_S) = (30, 0)$ against the interior pair $(12, 1)$, and the full count is dominated by the boundary. These are the accumulating subvarieties of Batyrev--Manin. In general $\alpha_S \geq Q_S := \sum_{i \in S} q_i$ by \cref{prop:a-bounds} applied to the stratum, so a proper stratum dominates the interior only when the Kummer sparsity of the interior is large enough.

\begin{cor}\label{cor:sparsity}
Let $n \geq 1$ and $\gcd(q_0, \dots, q_n) = 1$. Writing $Z^{\circ}_H(\bP^n(\Q), B)$ for the number of points of $\bP^n(\Q)$ with all coordinates nonzero and $H \leq B$,
\[
\frac{Z^{\circ}_{\h}\big( \WP^n_{\w}(\Q), X \big)}{Z^{\circ}_H\big( \bP^n(\Q), X^q \big)} \ =\ \frac{\zeta(n+1)\, c_{\w}}{2^n} \, X^{-q(n+1-a(\w))} (\log X)^{\beta(\w)} \big( 1 + o(1) \big) .
\]
The exponent $q(n+1-a(\w))$ is a nonnegative rational number; it is zero if and only if the $n_i$ are pairwise coprime, and it is positive for every well-formed $\w$ with $q \geq 2$.
\end{cor}

\begin{proof}
Schanuel's theorem applied to $\bP^n(\Q)$ and its coordinate hyperplanes gives $Z^{\circ}_H(\bP^n(\Q), B) = \frac{2^n}{\zeta(n+1)} B^{n+1} + O(B^{n} \log B)$; divide by \cref{thm:main} at $B = X^q$. The statements about the exponent are \cref{prop:a-bounds} and \cref{lem:wf-pairs}; rationality of $a(\w)$ holds because the linear program \cref{eq-LP} has integer data.
\end{proof}

Among the points of $\bP^n(\Q)$ of height $\leq B$ with nonzero coordinates, those that are images of rational points of the weighted space have relative density $\asymp B^{a(\w) - (n+1)} (\log B)^{\beta(\w)}$. For $(2,3,5)$ this is $B^{-13/5} \log B$; for $(2,2,3,3)$ it is $B^{-2} \log B$; for $(1,2,3,5)$ it is $B^{-109/30}$. The remaining points of the image are images of points of $\WP^n_{\w}(\bar\Q)$ with rational invariants but no rational representative, that is, of points whose field of definition in the sense of \cref{prop:fod} is a nontrivial Kummer extension of $\Q$ of degree at most $d_\phi$.

The proof of \cref{thm:main} occupies the remainder of this section. \cref{lem:L1} locates $M_{\w}$ inside the affine monoid $C_{\w}$ obtained by dropping the primitivity condition and computes $\operatorname{Hilb}(C_{\w})$; \cref{lem:L2} deduces that the local factor $\Phi_p$ differs from 
\[
\prod_{g \in T_{\w}} (1 - p^{-\langle g, s \rangle})^{-1}
\]
 by a factor $E_p$ with $|E_p - 1| \ll p^{-\mu}$; and \cref{lem:L3} globalizes this to 
 \[
 F = \prod_{g \in T_{\w}} \zeta(\langle g, s \rangle) \cdot G
 \]
  with $G$ holomorphic, bounded, and positive at the critical point. \cref{lem:lower} gives the unconditional lower bound 
  \[
  N_{\w}(B) \gg B^{a(\w)} (\log B)^{\beta(\w)}
  \]
   by parametrizing admissible tuples along the primal optimal face of \cref{eq-LP}. \cref{lem:L4} then applies the Tauberian theorem of de la Bret\`eche at a strictly complementary dual optimum, the lower bound pinning the logarithmic power to $\beta(\w)$, and \cref{prop:assembly} converts the tuple count into the point count.

\subsection{The lifting monoid and its Dirichlet series}\label{subsec-main-series}

An element of a monoid is \emph{irreducible} if it is nonzero and is not a sum of two nonzero elements. For $t \in \Z^{n+1}_{\geq 0}$ we write $z^t = z_0^{t_0} \cdots z_n^{t_n}$, and $e_i$ is the $i$-th standard basis vector.

\begin{lem}\label{lem:L1}
Let $\w$ be arbitrary weights, $n_i = q/q_i$, $g_{ij} = \gcd(n_i, n_j)$, and let
\[
\begin{split}
L_{\w}	&	= \left\{ t \in \Z^{n+1} \ : \ t_i \equiv t_j \pmod{g_{ij}} \ \text{ for all } i, j \right\}, \\
C_{\w}	&	= L_{\w} \cap \Z_{\geq 0}^{n+1},
\end{split}
\]
the \emph{lifting monoid}, so that $M_{\w} = C_{\w} \cap \{ \min_i t_i = 0 \}$ as in \cref{eq-monoid}. Write $\mathbf{1} = (1, \dots, 1)$ and let $\Gamma_{\w}$ be the graph on $\{0, \dots, n\}$ with an edge $\{i,j\}$ whenever $g_{ij} > 1$. Then:
\begin{enumerate}[label=\roman*)]
    \item $L_{\w}$ is a finite-index subgroup of $\Z^{n+1}$ containing $\mathbf{1}$ and all $n_i e_i$; the monoid $C_{\w}$ is finitely generated and positive, its real cone is the full orthant $\R_{\geq 0}^{n+1}$, and it has a unique minimal generating set $\operatorname{Hilb}(C_{\w})$, consisting of its irreducible elements.

    \item Every $t \in C_{\w}$ decomposes uniquely as $t = m \mathbf{1} + t'$ with $m \in \Z_{\geq 0}$ and $t' \in M_{\w}$, and then $m = \min_i t_i$. Consequently, as formal power series in $z = (z_0, \dots, z_n)$,
    \[
    \sum_{t \in C_{\w}} z^t = \frac{1}{1 - z^{\mathbf{1}}} \sum_{t \in M_{\w}} z^t .
    \]

    \item An element $h \in M_{\w} \setminus \{0\}$ is irreducible in $C_{\w}$ if and only if $h \in T_{\w}$, and every irreducible element of $C_{\w}$ lies in $M_{\w} \cup \{ \mathbf{1} \}$. Moreover $\mathbf{1}$ is irreducible if and only if $\Gamma_{\w}$ is connected. Hence
    \[
    \operatorname{Hilb}(C_{\w}) =
    \begin{cases}
    T_{\w} \cup \{ \mathbf{1} \}, & \Gamma_{\w} \text{ connected}, \\
    T_{\w}, & \Gamma_{\w} \text{ disconnected},
    \end{cases}
    \]
    and in the disconnected case $\mathbf{1} = \mathbf{1}_S + \mathbf{1}_{S^c}$ for any partition $\{0,\dots,n\} = S \sqcup S^c$ with no $\Gamma_{\w}$-edge between $S$ and $S^c$.

    \item Assume $n \geq 1$. For each $i$ the set $T_{\w}$ contains exactly one multiple of $e_i$, namely $\ell_i e_i$ with $\ell_i = \lcm_{j \neq i} g_{ij}$, and $\ell_i \mid n_i$.
\end{enumerate}
\end{lem}

\begin{proof}
i) $L_{\w}$ is the kernel of the homomorphism
\[
\Z^{n+1} \longrightarrow \prod_{i<j} \Z/g_{ij}, \qquad t \mapsto (t_i - t_j)_{i<j} ,
\]
hence a subgroup of finite index; it contains $\mathbf{1}$ since all differences vanish, and $n_i e_i$ since $g_{ij} \mid n_i$. Thus $C_{\w}$ is the intersection of a lattice with the rational polyhedral cone $\R_{\geq 0}^{n+1}$, so it is finitely generated by Gordan's lemma, and positive since $C_{\w} \cap (-C_{\w}) = \{0\}$; a positive affine monoid is generated by its irreducible elements, which form its unique minimal generating set \cite{bruns-gubeladze}. The real cone is the full orthant because $C_{\w}$ contains the axis multiples $n_i e_i$.

ii) Given $t \in C_{\w}$, set $m = \min_i t_i \geq 0$ and $t' = t - m\mathbf{1}$; then $t' \geq 0$ with $\min_i t'_i = 0$, and $t' \in L_{\w}$ because $L_{\w}$ is a group containing $\mathbf{1}$, so $t' \in M_{\w}$. Conversely $m \mathbf{1} + t' \in C_{\w}$ for every $m \in \Z_{\geq 0}$ and $t' \in M_{\w}$, and its minimum is $m$, so both parameters are determined by the sum. Summing $z^{m \mathbf{1} + t'}$ over the two parameters separately gives the series identity.

iii) Let $h \in M_{\w} \setminus \{0\}$ and fix $i_0$ with $h_{i_0} = 0$. If $h \notin T_{\w}$, pick nonzero $t \in M_{\w}$ with $t \leq h$, $t \neq h$; then
\[
h - t \in \Z_{\geq 0}^{n+1} \cap L_{\w} = C_{\w}
\]
is nonzero, so $h = t + (h - t)$ is reducible. If $h = u + v$ with $u, v \in C_{\w}$ nonzero, then $u_{i_0} = v_{i_0} = 0$, so $u \in M_{\w}$ is nonzero with $u \leq h$, $u \neq h$, and $h \notin T_{\w}$. This proves the equivalence.

Next, if $t \in C_{\w}$ has all coordinates positive and $t \neq \mathbf{1}$, then $t = \mathbf{1} + (t - \mathbf{1})$ with $t - \mathbf{1} \in C_{\w}$ nonzero by ii), so $t$ is reducible; hence every irreducible lies in $M_{\w} \cup \{\mathbf{1}\}$. Finally, a decomposition $\mathbf{1} = u + v$ with $u, v \in C_{\w}$ nonzero forces $u = \mathbf{1}_S$, $v = \mathbf{1}_{S^c}$ for a proper nonempty $S$, and $\mathbf{1}_S \in L_{\w}$ if and only if $g_{ij} = 1$ for all $i \in S$, $j \notin S$, that is, if and only if $S$ is a union of connected components of $\Gamma_{\w}$; such $S$ exists precisely when $\Gamma_{\w}$ is disconnected.

iv) Evaluated at $t = \ell e_i$, the congruences cutting out $L_{\w}$ read $\ell \equiv 0 \pmod{g_{ij}}$ for $j \neq i$, and $0 \equiv 0 \pmod{g_{jk}}$ for $j, k \neq i$; so the multiples of $e_i$ in $L_{\w}$ are exactly the multiples of $\ell_i = \lcm_{j \neq i} g_{ij}$, and $\ell_i \mid n_i$ because $g_{ij} \mid n_i$ for every $j$ and an lcm of divisors of $n_i$ divides $n_i$. Since $n \geq 1$, each $\ell_i e_i$ has vanishing minimum and hence lies in $M_{\w}$. A nonzero $t \in M_{\w}$ with $t \leq \ell_i e_i$ is of the form $\ell e_i$ with $0 < \ell \leq \ell_i$ and $\ell_i \mid \ell$, so $t = \ell_i e_i$; thus $\ell_i e_i \in T_{\w}$, while $\ell e_i$ with $\ell > \ell_i$ dominates it strictly and is not minimal.
\end{proof}

\begin{lem}\label{lem:L2}
Let $\w$, $n_i$, $g_{ij}$, $L_{\w}$, $C_{\w}$, $M_{\w}$, $T_{\w}$ be as in \cref{lem:L1}, let $D = \lcm_{i<j} g_{ij}$, and let
\[
R_{\w} = L_{\w} \cap \{0, 1, \dots, D-1\}^{n+1}
\]
be the set of residues modulo $D$ satisfying the congruences of $L_{\w}$, a finite set containing $0$. Then:
\begin{enumerate}[label=\roman*)]
    \item As formal power series,
    \[
    \begin{split}
    \sum_{t \in C_{\w}} z^t	&	= \frac{\sum_{r \in R_{\w}} z^r}{\prod_{i=0}^{n} (1 - z_i^{D})} , \\
    \sum_{t \in M_{\w}} z^t	&	= \left( 1 - z^{\mathbf{1}} \right) \frac{\sum_{r \in R_{\w}} z^r}{\prod_{i=0}^{n} (1 - z_i^{D})} .
    \end{split}
    \]

    \item For a prime $p$ and $s \in \C^{n+1}$ with $\Re s_i > 0$ set
    \[
    \begin{split}
    \Phi_p(s)	&	= \sum_{t \in M_{\w}} p^{-\langle t, s \rangle} , \\
    E_p(s)	&	= \Phi_p(s) \prod_{g \in T_{\w}} \left( 1 - p^{-\langle g, s \rangle} \right) ,
    \end{split}
    \]
    both series converging absolutely. Let the integers $a_u$, indexed by $u \in \Z_{\geq 0}^{n+1}$, be defined by the formal identity
    \[
    \left( \sum_{t \in M_{\w}} z^t \right) \prod_{g \in T_{\w}} \left( 1 - z^{g} \right) = \sum_{u \in \Z_{\geq 0}^{n+1}} a_u \, z^u
    \]
    in $\Z[[z_0, \dots, z_n]]$, so that $E_p(s) = \sum_u a_u \, p^{-\langle u, s \rangle}$ for every $p$; in particular the $a_u$ do not depend on $p$. Then $a_0 = 1$, one has $|a_u| \leq 2^{|T_{\w}|}$ for every $u$, and $a_u = 0$ unless $u = 0$ or
    \[
    u \geq g + h \quad \text{coordinatewise, for some } g \in T_{\w} \text{ and some } h \in \operatorname{Hilb}(C_{\w}) .
    \]

    \item Assume $n \geq 1$. For every $\sigma_0 > 0$ there is a constant $A = A(\w, \sigma_0)$ such that for every prime $p$ and every $s \in \C^{n+1}$ with $\Re s_i \geq \sigma_0$ for all $i$,
    \[
    \left| E_p(s) - 1 \right| \leq A \, p^{-\mu(\Re s)},
    \qquad
    \mu(\sigma) = \min_{g \in T_{\w}} \langle g, \sigma \rangle + \min_{h \in T_{\w} \cup \{\mathbf{1}\}} \langle h, \sigma \rangle .
    \]
    Moreover $E_p(\sigma) > 0$ for all real $\sigma$ with $\sigma_i \geq \sigma_0$.
\end{enumerate}
\end{lem}

\begin{proof}
i) Since $g_{ij} \mid D$ for all $i, j$, we have $D \Z^{n+1} \subseteq L_{\w}$, so membership in $L_{\w}$ depends only on the residue of $t$ modulo $D$ coordinatewise. For $r \in \{0, \dots, D-1\}^{n+1}$, the set of $t \in \Z_{\geq 0}^{n+1}$ with $t \equiv r \pmod{D}$ coordinatewise is exactly $r + D \Z_{\geq 0}^{n+1}$, and it lies in $L_{\w}$ if and only if $r \in R_{\w}$. Summing $z^t$ over one such translated orthant gives
\[
\sum_{w \in \Z_{\geq 0}^{n+1}} z^{\,r + D w} = \frac{z^r}{\prod_{i=0}^{n} (1 - z_i^{D})} ,
\]
and summing over $r \in R_{\w}$ gives the first identity. The second follows from the first by \cref{lem:L1}~ii).

ii) Absolute convergence is clear from $M_{\w} \subseteq \Z_{\geq 0}^{n+1}$ and $\sum_{t \geq 0} p^{-\langle t, \Re s \rangle} = \prod_i (1 - p^{-\Re s_i})^{-1} < \infty$, and evaluating the formal identity at $z_i = p^{-s_i}$ gives $E_p(s) = \sum_u a_u \, p^{-\langle u, s \rangle}$. Expanding the product over the subsets of $T_{\w}$ and writing $g_S = \sum_{g \in S} g$,
\[
\sum_{u} a_u \, z^u = \sum_{t \in M_{\w}} \ \sum_{S \subseteq T_{\w}} (-1)^{|S|} z^{\, t + g_S} ,
\qquad \text{so} \qquad
a_u = \sum_{(t, S)} (-1)^{|S|} ,
\]
the last sum over the pairs $(t, S)$ with $t \in M_{\w}$, $S \subseteq T_{\w}$ and $t + g_S = u$. Each $S$ determines $t = u - g_S$, so at most $2^{|T_{\w}|}$ pairs contribute to any given $u$ and $|a_u| \leq 2^{|T_{\w}|}$.

Call a contributing pair \emph{special} if $S = \emptyset$ and $t \in T_{\w} \cup \{0\}$, or if $t = 0$ and $S$ is a singleton. We claim that a non-special contributing pair forces $u \geq g + h$ with $g \in T_{\w}$ and $h \in \operatorname{Hilb}(C_{\w})$. There are three cases. If $t \notin T_{\w} \cup \{0\}$, then $t \geq g$ for some $g \in T_{\w}$ by Dickson's lemma, with $g \neq t$, so that
\[
c = t - g \in \Z_{\geq 0}^{n+1} \cap L_{\w} = C_{\w}
\]
is nonzero; since $\operatorname{Hilb}(C_{\w})$ generates $C_{\w}$, we have $c \geq h$ for some $h \in \operatorname{Hilb}(C_{\w})$, whence $u \geq t = g + c \geq g + h$. If $t \in T_{\w}$ and $S \neq \emptyset$, then $u \geq t + g_1$ for any $g_1 \in S$. If $t = 0$ and $|S| \geq 2$, then $u \geq g_1 + g_2$ for two distinct $g_1, g_2 \in S$. In the last two cases both summands lie in $T_{\w} \subseteq \operatorname{Hilb}(C_{\w})$ by \cref{lem:L1}~iii), and the claim follows.

It remains to evaluate $a_u$. For $u = 0$ the only contributing pair is $(0, \emptyset)$, so $a_0 = 1$. For $u \in T_{\w}$ no non-special pair contributes: such a pair would give $u \geq g + h$ with $h \neq 0$, hence $u \geq g$ and $u \neq g$ for some $g \in T_{\w} \subseteq M_{\w} \setminus \{0\}$, contradicting the minimality of $u$ in $M_{\w} \setminus \{0\}$. The special pairs contributing to $u \in T_{\w}$ are $(u, \emptyset)$ and $(0, \{u\})$, contributing $+1$ and $-1$, so $a_u = 0$. Finally, if $u \neq 0$ and $u \notin T_{\w}$, then no special pair contributes at all, so every contributing pair is non-special and $u \geq g + h$ as claimed.

iii) Since $n \geq 1$ we have $T_{\w} \neq \emptyset$. Write $\sigma = \Re s$. By ii) and $|p^{-\langle u, s \rangle}| = p^{-\langle u, \sigma \rangle}$,
\[
\left| E_p(s) - 1 \right| \leq \sum_{u \neq 0} |a_u| \, p^{-\langle u, \sigma \rangle} \leq 2^{|T_{\w}|} \sum_{u} p^{-\langle u, \sigma \rangle} ,
\]
the last sum over those $u$ dominating $g + h$ for some $g \in T_{\w}$ and $h \in \operatorname{Hilb}(C_{\w})$. Writing $u = g + h + w$ with $w \in \Z_{\geq 0}^{n+1}$ and summing the geometric series in $w$, each such $u$ being counted at least once,
\[
\left| E_p(s) - 1 \right|
\leq 2^{|T_{\w}|} \sum_{g \in T_{\w}} \ \sum_{h \in \operatorname{Hilb}(C_{\w})} p^{-\langle g + h, \sigma \rangle} \prod_{i=0}^n \left( 1 - p^{-\sigma_i} \right)^{-1} .
\]
By \cref{lem:L1}~iii) we have $\operatorname{Hilb}(C_{\w}) \subseteq T_{\w} \cup \{\mathbf{1}\}$, so the double sum has at most $|T_{\w}| \left( |T_{\w}| + 1 \right)$ terms, each satisfying $\langle g + h, \sigma \rangle \geq \mu(\sigma)$; and $p^{-\sigma_i} \leq 2^{-\sigma_0}$ gives $\prod_i (1 - p^{-\sigma_i})^{-1} \leq (1 - 2^{-\sigma_0})^{-(n+1)}$. Hence the bound holds with
\[
A(\w, \sigma_0) = 2^{|T_{\w}|} \, |T_{\w}| \left( |T_{\w}| + 1 \right) \left( 1 - 2^{-\sigma_0} \right)^{-(n+1)} .
\]
Positivity is immediate: $\Phi_p(\sigma) \geq 1$ because $0 \in M_{\w}$, and $0 < 1 - p^{-\langle g, \sigma \rangle} < 1$ for every $g \in T_{\w}$ because $\langle g, \sigma \rangle > 0$; so $E_p(\sigma) > 0$.
\end{proof}

\begin{lem}\label{lem:L3}
Assume $n \geq 1$, let $\w$, $T_{\w}$, $\Phi_p$, $E_p$, $\mu$ be as in \cref{lem:L2}, and let $\ell_i$ be the axis generators of \cref{lem:L1}~iv). Set
\[
\kappa = \sum_{i=0}^{n} \frac{1}{\ell_i}, \qquad \eta = \frac{1}{2} \min(1, \kappa), \qquad \delta_{\w} = \frac{\min(1,\kappa)}{2\,(1 + \min(1,\kappa))} ,
\]
so that $0 < \delta_{\w} \leq \tfrac{1}{4}$, and for $0 < \delta \leq \delta_{\w}$ define the open tube
\[
\Omega_\delta = \left\{ s \in \C^{n+1} \ : \ \langle g, \Re s \rangle > 1 - \delta \ \text{ for all } g \in T_{\w} \right\} .
\]
Let
\[
F(s) = \sum_{y} \ \prod_{i=0}^{n} y_i^{-s_i} ,
\]
the sum over the tuples $y \in \Z_{\geq 1}^{n+1}$ whose valuation vector $\big( v_p(y_0), \dots, v_p(y_n) \big)$ lies in $M_{\w}$ for every prime $p$; these tuples are primitive, since every element of $M_{\w}$ has vanishing minimum. Then:
\begin{enumerate}[label=\roman*)]
    \item $F(s)$ converges absolutely on $\{ \langle g, \Re s \rangle > 1 \ \text{for all } g \in T_{\w} \}$, and there
    \[
    F(s) = \prod_p \Phi_p(s) = \prod_{g \in T_{\w}} \zeta\big( \langle g, s \rangle \big) \cdot G(s), \qquad G(s) = \prod_p E_p(s) .
    \]

    \item Every $s \in \Omega_{\delta_{\w}}$ satisfies $\Re s_i > 3 / (4 \max_j \ell_j)$ for all $i$, and $\mu(\Re s) > 1 + \eta$. Consequently the product $G(s)$ converges absolutely and uniformly on $\Omega_{\delta_{\w}}$, defines a holomorphic function there, and satisfies 
    \[
    |G(s)| \leq C_G(\w)
    \]
     uniformly on $\Omega_{\delta_{\w}}$, with a constant depending only on $\w$.

    \item $F$ extends meromorphically to $\Omega_{\delta_{\w}}$, the extension being given throughout $\Omega_{\delta_{\w}}$ by
    \[
    F(s) = \prod_{g \in T_{\w}} \zeta\big( \langle g, s \rangle \big) \cdot G(s)
    \]
    with $G$ holomorphic and bounded there. Its polar divisor in $\Omega_{\delta_{\w}}$ is therefore contained in the union of the hyperplanes $\langle g, s \rangle = 1$, $g \in T_{\w}$, and 
    \[
    |F(s)| \leq C_G(\w) \prod_{g \in T_{\w}} \left| \zeta(\langle g, s \rangle) \right|, \quad \text{  on } \quad \Omega_{\delta_{\w}}.
    \]

    \item Let $\sigma^*$ be any optimal solution of the dual of \cref{eq-LP}, so that $\sigma^* \geq 0$ and $\langle g, \sigma^* \rangle \geq 1$ for all $g \in T_{\w}$. Then 
    \[
    \sum_i \sigma^*_i = a(\w), \quad \text{  and } \quad \sigma^* \in \Omega_{\delta_{\w}}
    \]
     with $\sigma^*_i \geq 1/\ell_i > 0$ for every $i$ and $G(\sigma^*) > 0$.
\end{enumerate}
\end{lem}

\begin{proof}
We prove ii) first, since i) rests on the convergence of $G$.

ii) Let $\sigma = \Re s$ with $\langle g, \sigma \rangle > 1 - \delta_{\w}$ for all $g \in T_{\w}$. Taking $g = \ell_i e_i$ gives
\[
\sigma_i > \frac{1 - \delta_{\w}}{\ell_i} \geq \frac{1 - \delta_{\w}}{\max_j \ell_j} \geq \frac{3}{4 \max_j \ell_j} =: \sigma_0(\w) > 0 ,
\]
using $\delta_{\w} \leq \tfrac14$; in particular $\sigma$ lies in the positive orthant, where \cref{lem:L2} applies, with a threshold $\sigma_0(\w)$ depending only on $\w$. Summing over $i$ gives $\langle \mathbf{1}, \sigma \rangle > (1 - \delta_{\w}) \kappa$, so that the second minimum in $\mu$, taken over $T_{\w} \cup \{\mathbf{1}\}$, exceeds $(1 - \delta_{\w}) \min(1, \kappa)$, whence
\[
\begin{split}
\mu(\sigma)	&	= \min_{g \in T_{\w}} \langle g, \sigma \rangle + \min_{h \in T_{\w} \cup \{\mathbf{1}\}} \langle h, \sigma \rangle \\
		&	> (1 - \delta_{\w}) + (1 - \delta_{\w}) \min(1, \kappa) \\
		&	= (1 - \delta_{\w}) \big( 1 + \min(1, \kappa) \big) = 1 + \tfrac{1}{2} \min(1,\kappa) = 1 + \eta ,
\end{split}
\]
the last two equalities by the choice of $\delta_{\w}$ and of $\eta$. Writing $A = A(\w, \sigma_0(\w))$, \cref{lem:L2}~iii) now gives
\[
\left| E_p(s) - 1 \right| \leq A \, p^{-\mu(\Re s)} \leq A \, p^{-(1+\eta)}  \qquad \text{for all } s \in \Omega_{\delta_{\w}} ,
\]
so 
\[
\sum_p \sup_{\Omega_{\delta_{\w}}} |E_p(s) - 1| \leq A \sum_p p^{-(1+\eta)} < \infty
\]
 since $\eta > 0$. Each $E_p$ is holomorphic on $\Omega_{\delta_{\w}}$, the series $\Phi_p$ converging absolutely there; hence the product $G$ converges absolutely and uniformly, is holomorphic, and
\[
|G(s)| \leq \prod_p \left( 1 + A \, p^{-(1+\eta)} \right) =: C_G(\w) < \infty .
\]

i) The stated region lies in $\Omega_{\delta_{\w}}$, and on it each $\zeta(\langle g, s \rangle)$ is given by its absolutely convergent Euler product. The coefficients of $F$ are nonnegative, so absolute convergence may be tested at $\sigma = \Re s$, where unique factorization gives 
\[
\sum_y \prod_i y_i^{-\sigma_i} = \prod_p \Phi_p(\sigma)
\]
 in $[0, \infty]$; multiplying the local identities
\[
\Phi_p(s) = E_p(s) \prod_{g \in T_{\w}} \left( 1 - p^{-\langle g, s \rangle} \right)^{-1}
\]
of \cref{lem:L2}~ii) over $p$ exhibits the right side as $\prod_{g \in T_{\w}} \zeta(\langle g, \sigma \rangle) \cdot G(\sigma)$, both factors being finite, the first by $\langle g, \sigma \rangle > 1$ and the second by ii). The same computation at $s$ gives the factorization.

iii) The region of i) is contained in $\Omega_{\delta_{\w}}$, since $1 > 1 - \delta_{\w}$, and on it 
\[
F = \prod_{g \in T_{\w}} \zeta(\langle g, s \rangle) \cdot G
\]
 by i). The right side is meromorphic on $\Omega_{\delta_{\w}}$: each factor $\zeta(\langle g, s \rangle)$ is the quotient of the entire function $\big( \langle g, s \rangle - 1 \big) \zeta(\langle g, s \rangle)$ by the linear form $\langle g, s \rangle - 1$, and $G$ is holomorphic and bounded there by ii). As an intersection of open half-spaces in the coordinates $\Re s$, the tube $\Omega_{\delta_{\w}}$ is convex, hence connected, so the right side is the unique meromorphic extension of $F$. The two remaining assertions follow, the second from $|G| \leq C_G(\w)$. Note that $G$ is not asserted to be zero-free on $\Omega_{\delta_{\w}}$, so a pole of a zeta factor may cancel; the containment cannot be improved to an equality without further work.

iv) Strong duality gives 
\[
\sum_i \sigma^*_i = a(\w).
\]
 Dual feasibility gives $\langle g, \sigma^* \rangle \geq 1 > 1 - \delta_{\w}$, so $\sigma^* \in \Omega_{\delta_{\w}}$, and against the axis generators $\ell_i \sigma^*_i \geq 1$, that is $\sigma^*_i \geq 1/\ell_i > 0$. By ii), 
 \[
 \sum_p |E_p(\sigma^*) - 1| < \infty;
 \]
  each factor satisfies $E_p(\sigma^*) > 0$ by \cref{lem:L2}~iii), and a product $\prod_p E_p$ with $\sum_p |E_p - 1| < \infty$ and no vanishing factor converges to a nonzero limit, positive here since every factor is.
\end{proof}

\subsection{The lower bound}\label{subsec-main-lower}

\begin{lem}\label{lem:lower}
Let $n \geq 1$, let $\w = (q_0, \dots, q_n)$ with $\gcd(q_0, \dots, q_n) = 1$, let $N_{\w}(B)$ denote the number of tuples $y \in \Z_{\geq 1}^{n+1}$ with $\max_i y_i \leq B$ whose valuation vector at every prime lies in $M_{\w}$, and let $a(\w)$ and $\beta(\w)$ be the value and the dimension of the optimal face of \cref{eq-LP}. Then, for $B \geq 2$,
\[
N_{\w}(B) \gg B^{\,a(\w)} (\log B)^{\beta(\w)} .
\]
\end{lem}

\begin{proof}
By the theorem of Goldman--Tucker on strict complementarity \cite[\S 7.9]{schrijver1986}, there is an optimal primal--dual pair $(c^*, \sigma^*)$ of \cref{eq-LP} with $\sigma^*$ in the relative interior of the dual optimal face, such that $c^*_g > 0$ exactly for $g$ in the tight set 
\[
T^{\circ} = \{ g \in T_{\w} : \langle g, \sigma^* \rangle = 1 \}
\]
 and, for each coordinate $i$, either $\sigma^*_i > 0$ or the $i$-th primal constraint is slack. By \cref{lem:L3}~iv), $\sigma^*_i \geq 1/\ell_i > 0$ for every $i$, so every primal constraint is tight at $c^*$:
\begin{equation}\label{eq-strict-comp}
\sum_{g \in T^{\circ}} c^*_g \, g = \mathbf{1}, \qquad c^*_g > 0 \quad (g \in T^{\circ}),
\end{equation}
and $\sum_i \sigma^*_i = \sum_g c^*_g = a(\w)$ by strong duality. Write $\varrho = |T^{\circ}|$ and $r = \operatorname{rank}(T^{\circ})$. Complementary slackness against $\sigma^*$ characterizes the primal optimal face as
\[
\Phi = \left\{ c \in \R_{\geq 0}^{T_{\w}} \ : \ c_g = 0 \ \text{for } g \notin T^{\circ}, \quad \textstyle\sum_{g \in T^{\circ}} c_g \, g = \mathbf{1} \right\},
\]
using again that every $\sigma^*_i > 0$ forces every coordinate constraint tight. The affine solution set of $\sum_{g \in T^{\circ}} c_g \, g = \mathbf{1}$ has dimension $\varrho - r$, and $\Phi$ contains the strictly positive point $c^*$ of \cref{eq-strict-comp}, an interior point of the nonnegativity constraints; hence
\begin{equation}\label{eq-face-dim}
\dim \Phi = \varrho - r = \beta(\w) .
\end{equation}
Set $d = \beta(\w)$ and choose linearly independent $u^{(1)}, \dots, u^{(d)} \in \R^{T^{\circ}}$ with $\sum_{g \in T^{\circ}} u^{(k)}_g \, g = \mathbf{0}$; these span the direction space of $\Phi$. Pairing with $\sigma^*$ gives
\[
\sum_g u^{(k)}_g = \big\langle \sum_g u^{(k)}_g \, g, \sigma^* \big\rangle = 0,
\]
so the objective of \cref{eq-LP} is constant along these directions, as it must be on the optimal face.

To a tuple $z = (z_g)_{g \in T^{\circ}}$ of pairwise coprime positive integers we associate $y(z) \in \Z_{\geq 1}^{n+1}$ with coordinates 
\[
y_i(z) = \prod_{g \in T^{\circ}} z_g^{g_i}.
\]
 At a prime $p$ at most one $z_g$ is divisible by $p$, so the valuation vector of $y(z)$ at $p$ is either $\mathbf{0}$ or $v_p(z_g) \, g$ for that single $g$; both lie in $M_{\w}$, since $\mathbf{0} \in M_{\w}$ and $M_{\w}$ is stable under positive integer multiples, the congruences of \cref{eq-monoid} being linear and $\min_i e g_i = e \min_i g_i = 0$. Hence every $y(z)$ is counted by $N_{\w}(B)$ as soon as $\max_i y_i(z) \leq B$. Moreover distinct elements of $T_{\w}$ are never proportional: if $g' = \lambda g$ with $\lambda > 1$, then $g'$ dominates the nonzero element $g$ of $M_{\w}$ coordinatewise with $g' \neq g$, contradicting the minimality of $g'$. Consequently the decomposition $v_p(y) = e \, g$ with 
 \[
 e \in \Z_{\geq 1}
 \]
  and $g \in T^{\circ}$ is unique when it exists, the tuple $z$ is recovered from $y(z)$ by $z_g = \prod_{p} p^{\, e_p}$, the product over the primes $p$ with $v_p(y) = e_p \, g$ for some $e_p \in \Z_{\geq 1}$, and $z \mapsto y(z)$ is injective on pairwise coprime tuples.

Assume for the moment that $d \geq 1$. Let 
\[
m_0 = \min \big\{ \| \sum_k m_k u^{(k)} \|_{\infty} : m \in \Z^d \setminus \{0\} \big\},
\]
 which is positive since the $u^{(k)}$ are linearly independent and the minimum over a lattice is attained; put $\tau = (\log 4)/m_0$, and fix $\varepsilon > 0$ so small that 
 \[
 \tau \varepsilon d \max_{k, g} |u^{(k)}_g| \leq \tfrac{1}{2} \min_{g \in T^{\circ}} c^*_g.
 \]
  For $j \in \Z^d$ with $\|j\|_{\infty} \leq \varepsilon \log B$ set
\[
c(j) = c^* + \frac{\tau}{\log B} \sum_{k=1}^{d} j_k \, u^{(k)} ,
\]
so that $c(j)_g \geq \tfrac{1}{2} c^*_g > 0$ and $c(j) \in \Phi$; in particular 
\[
\sum_g c(j)_g \, g = \mathbf{1}
\]
 and 
 \[
 \sum_g c(j)_g = a(\w).
 \]
  Consider the boxes
\[
R_j = \prod_{g \in T^{\circ}} \left( \tfrac{1}{2} B^{c(j)_g}, \; B^{c(j)_g} \right] .
\]
For $z \in R_j$ one has $y_i(z) \leq B^{\sum_g c(j)_g g_i} = B$ for every $i$. For $j \neq j'$ the choice of $\tau$ produces a coordinate $g$ with $|c(j)_g - c(j')_g| \log B \geq \log 4$, so the corresponding intervals are disjoint and $R_j \cap R_{j'} = \emptyset$. If instead $d = 0$, then $\Phi = \{c^*\}$ and the construction degenerates to the single box 
\[
R = \prod_{g \in T^{\circ}} ( \tfrac{1}{2} B^{c^*_g}, B^{c^*_g} ],
\]
 no $m_0$, $\tau$ or $\varepsilon$ being needed; the count below applies to it verbatim.

Fix $P_0 = P_0(\varrho)$ so large that 
\[
4 \varrho^2 \sum_{p > P_0} p^{-2} \leq \tfrac{1}{2} \, 2^{-\varrho} \prod_{p \leq P_0} (1 - 1/p)^{\varrho};
\]
 such a $P_0$ exists because the left side is $O(1/P_0)$ while the right side is $\gg (\log P_0)^{-\varrho}$ by Mertens' theorem. Within $R_j$, the tuples with every $z_g$ coprime to $\prod_{p \leq P_0} p$ number at least 
 \[
 2^{-\varrho} \prod_{p \leq P_0}(1 - 1/p)^{\varrho} \prod_g B^{c(j)_g} + O\big( B^{a(\w) - \theta_0} \big)
 \]
  for some $\theta_0 > 0$, by counting reduced residues in each interval separately, the sides being $\gg B^{c^*_g/2}$; among these, the tuples in which two coordinates $z_g, z_{g'}$ share a prime factor, necessarily $p > P_0$, number at most 
  \[
  \sum_{g \neq g'} \sum_{p > P_0} 4 \, p^{-2} \prod_h B^{c(j)_h} \leq 4 \varrho^2 \big( \sum_{p > P_0} p^{-2} \big) B^{a(\w)},
  \]
   which by the choice of $P_0$ is at most half of the main term above. Hence each box contributes $\gg B^{a(\w)}$ pairwise coprime tuples; the boxes are disjoint, the map $z \mapsto y(z)$ is injective, and summing over the $\gg (\log B)^d$ admissible $j$ gives $N_{\w}(B) \gg B^{a(\w)} (\log B)^{\beta(\w)} .$
\end{proof}

\subsection{Proof of the Main Theorem}\label{subsec-main-proof}

\begin{lem}\label{lem:L4}
Let $n \geq 1$, let $\w = (q_0, \dots, q_n)$ with $\gcd(q_0, \dots, q_n) = 1$, and let $N_{\w}(B)$ be as in \cref{lem:lower}. Then there exist $\theta > 0$ and a polynomial $\Psi \in \R[x]$ of exact degree $\beta(\w)$, with positive leading coefficient $c_{\Psi}$, such that
\[
N_{\w}(B) = B^{\,a(\w)} \, \Psi(\log B) + O\!\left( B^{\,a(\w) - \theta} \right) .
\]
\end{lem}

\begin{proof}
We use \cite[Th\'eor\`eme~1]{delaBreteche2001}, which we state in the case needed here. Let $F(s) = \sum_{y \in \Z^{n+1}_{\geq 1}} f(y) \prod_i y_i^{-s_i}$ with $f \geq 0$, and suppose there are $c \in \R^{n+1}_{\geq 0}$, a finite family $\ell_1, \dots, \ell_\varrho$ of nonzero linear forms with nonnegative coefficients, a finite family $h_1, \dots, h_{\varrho'}$ of linear forms with nonnegative coefficients, and $\delta_1, \delta_2, \delta_3 > 0$ such that:
\begin{enumerate}[label=(P\arabic*)]
\item $F$ converges absolutely for $\Re s_i > c_i$ $(0 \leq i \leq n)$;
\item $H(s) := F(s + c) \prod_{j=1}^{\varrho} \ell_j(s)$ extends holomorphically to
\[
\mathcal{D}(\delta_1, \delta_3) = \big\{ s \in \C^{n+1} : \Re \ell_j(s) > -\delta_1 \ (1 \leq j \leq \varrho), \ \Re h_j(s) > -\delta_3 \ (1 \leq j \leq \varrho') \big\} ;
\]
\item for all $\varepsilon, \varepsilon' > 0$, uniformly on $\mathcal{D}(\delta_1 - \varepsilon, \delta_3 - \varepsilon')$,
\[
|H(s)| \ll \prod_{j=1}^{\varrho} \big( |\Im \ell_j(s)| + 1 \big)^{1 - \delta_2 \min\{0, \Re \ell_j(s)\}} \Big( 1 + \sum_i |\Im s_i| \Big)^{\varepsilon} .
\]
\end{enumerate}
Let $w$ be the number of indices $j$ with $c_j = 0$. Then for every $\beta \in (0, \infty)^{n+1}$ there are a polynomial $\Psi_\beta$ of degree at most $\varrho + w - \operatorname{rank}(\ell_1, \dots, \ell_\varrho)$ and $\theta > 0$ such that $\sum_{y_i \leq B^{\beta_i} \ \forall i} f(y) = B^{\langle c, \beta \rangle} \Psi_\beta(\log B) + O(B^{\langle c, \beta \rangle - \theta})$. We apply this with $\beta = \mathbf{1}$, with $c = \sigma^*$, which has all coordinates positive by \cref{lem:L3}~iv), so that $w = 0$, and with $h_j$ the coordinate forms $s \mapsto s_j$.

We verify (P1)--(P3) for the series $F$ of \cref{lem:L3}, with the notation of \cref{lem:lower}: $(c^*, \sigma^*)$ the strictly complementary pair of \cref{eq-strict-comp}, $T^{\circ}$ its tight set, $\varrho = |T^{\circ}|$, $r = \operatorname{rank}(T^{\circ})$, and polar forms $\ell_g = \langle g, \cdot \rangle$ for $g \in T^{\circ}$.

\emph{(P1).} Dual feasibility gives $\langle g, \sigma^* \rangle \geq 1$ for all $g \in T_{\w}$, so any real $\sigma$ with $\sigma_i > \sigma^*_i$ for every $i$ satisfies $\langle g, \sigma \rangle > \langle g, \sigma^* \rangle \geq 1$ (every $g$ is nonzero with nonnegative entries) and lies in the region of \cref{lem:L3}~i), where
\[
F(\sigma) = \prod_{g \in T_{\w}} \zeta\big( \langle g, \sigma \rangle \big) \cdot G(\sigma) < \infty .
\]
Since $F$ has nonnegative coefficients, it converges absolutely at every $s$ with $\Re s_i > \sigma^*_i$ coordinatewise.

\emph{(P2).} By \cref{lem:L3}~iii), on $\Omega_{\delta_{\w}} - \sigma^*$ we have
\[
\begin{split}
H(s)	&	= F(s + \sigma^*) \prod_{g \in T^{\circ}} \langle g, s \rangle \\
	&	= \prod_{g \in T^{\circ}} \zeta\big( \langle g, s \rangle + 1 \big) \langle g, s \rangle
		\cdot \prod_{g \in T_{\w} \setminus T^{\circ}} \zeta\big( \langle g, s \rangle + \langle g, \sigma^* \rangle \big)
		\cdot G(s + \sigma^*) ,
\end{split}
\]
where $\langle g, \sigma^* \rangle = 1$ for $g \in T^{\circ}$ and $\langle g, \sigma^* \rangle > 1$ for $g \notin T^{\circ}$. Each factor of the first product is entire, the simple pole of $\zeta$ at $1$ being cancelled. Put $\delta_1 = \delta_{\w}$ and
\[
\delta_3 = \min\Big( \delta_{\w}, \ \min_{g \in T_{\w} \setminus T^{\circ}} \frac{\langle g, \sigma^* \rangle - 1}{2 \, \langle g, \mathbf{1} \rangle} \Big) > 0 ,
\]
the inner minimum being over a finite set and omitted if $T_{\w} = T^{\circ}$. Let $s \in \mathcal{D}(\delta_1, \delta_3)$. For $g \in T^{\circ}$ we have $\langle g, \Re s + \sigma^* \rangle > 1 - \delta_{\w}$ directly. For $g \in T_{\w} \setminus T^{\circ}$, since $g \geq 0$ and $\Re s_i > -\delta_3$,
\[
\langle g, \Re s + \sigma^* \rangle \ >\ \langle g, \sigma^* \rangle - \delta_3 \langle g, \mathbf{1} \rangle \ \geq\ \langle g, \sigma^* \rangle - \tfrac12 \big( \langle g, \sigma^* \rangle - 1 \big) \ =\ 1 + \tfrac12 \big( \langle g, \sigma^* \rangle - 1 \big) \ >\ 1 .
\]
Hence $s + \sigma^* \in \Omega_{\delta_{\w}}$, so $G(s + \sigma^*)$ is holomorphic and bounded on $\mathcal{D}(\delta_1, \delta_3)$ by \cref{lem:L3}~ii), and each factor $\zeta(\langle g, s \rangle + \langle g, \sigma^* \rangle)$ with $g \notin T^{\circ}$ is holomorphic there, its argument having real part $> 1$. This is (P2).

\emph{(P3).} On $\mathcal{D}(\delta_1, \delta_3)$ the factor $G(s + \sigma^*)$ is bounded, and so is every factor of the second product, since $\Re\big( \langle g, s \rangle + \langle g, \sigma^* \rangle \big) > 1 + \delta'_3$ there, with $\delta'_3 = \tfrac12 \min_{g \notin T^{\circ}} (\langle g, \sigma^* \rangle - 1)$, gives
\[
\left| \zeta\big( \langle g, s \rangle + \langle g, \sigma^* \rangle \big) \right| \leq \zeta(1 + \delta'_3) .
\]
For $g \in T^{\circ}$ the corresponding factor is $(w - 1)\zeta(w)$ evaluated at $w = \langle g, s \rangle + 1$, an entire function; writing $t = \Im w$, the classical estimate $\zeta(\tfrac12 + it) \ll_{\varepsilon} (1 + |t|)^{1/6 + \varepsilon}$ and the bound $\zeta(1 + it) \ll \log(2 + |t|)$ give, for every $t$,
\[
\begin{split}
\left| (w-1)\zeta(w) \right|	&	\ll_{\varepsilon} (1 + |t|)^{\frac{7}{6} + \varepsilon}	\qquad (\Re w = \tfrac{1}{2}), \\
\left| (w-1)\zeta(w) \right|	&	\ll_{\varepsilon} (1 + |t|)^{1 + \varepsilon}		\qquad\ \ (\Re w = 1) .
\end{split}
\]
The pole having been cleared, the Phragm\'en--Lindel\"of principle applies to $(w-1)\zeta(w)$ and interpolates these across the strip $\tfrac12 \leq \Re w \leq 1$ to $\ll_{\varepsilon} (1 + |t|)^{1 + (1 - \Re w)/3 + \varepsilon}$, while on $1 \leq \Re w \leq C$ it gives $\ll_{\varepsilon} (1 + |t|)^{1 + \varepsilon}$ for any fixed $C$. Since $\delta_{\w} \leq \tfrac14$, the domain above has $\Re \langle g, s \rangle > -\tfrac12$ and hence $\Re w > \tfrac12$, so that, uniformly for $\Re s$ in a bounded set,
\[
\zeta\big( \langle g, s \rangle + 1 \big) \langle g, s \rangle
\ll_{\varepsilon} \Big( \big| \Im \langle g, s \rangle \big| + 1 \Big)^{1 - \frac{1}{3} \min\{ 0, \, \Re \langle g, s \rangle \} + \varepsilon} .
\]
Multiplying the factors and absorbing the $\varepsilon$'s into the last factor of the majorization gives (P3) with $\delta_2 = 1/3$, uniformly on all of $\mathcal{D}(\delta_1, \delta_3)$, since the bounds above are uniform for $\Re \langle g, s \rangle > -\tfrac12$ and $G(s + \sigma^*)$ is bounded on the whole domain.

Th\'eor\`eme~1 with $\beta = \mathbf{1}$ and $w = 0$ now yields
\[
N_{\w}(B) = B^{\langle \sigma^*, \mathbf{1} \rangle} \, \Psi(\log B) + O\!\left( B^{\langle \sigma^*, \mathbf{1} \rangle - \theta} \right)
\]
with $\langle \sigma^*, \mathbf{1} \rangle = a(\w)$ and $\deg \Psi \leq \varrho - r$, and $\varrho - r = \beta(\w)$ by \cref{eq-face-dim}. Were $\Psi$ of degree strictly less than $\beta(\w)$, or of degree $\beta(\w)$ with negative leading coefficient, we should have
\[
\limsup_{B \to \infty} \ N_{\w}(B) \, B^{-a(\w)} (\log B)^{-\beta(\w)} \leq 0 ,
\]
contradicting \cref{lem:lower}. Hence $\deg \Psi = \beta(\w)$ exactly and $c_{\Psi} > 0$.
\end{proof}
\begin{rem}\label{rem-nondegenerate}
\cite[Th\'eor\`eme~1]{delaBreteche2001} bounds $\deg \Psi$ from above but does not evaluate it, which is why \cref{lem:lower} enters. The alternative is \cite[Th\'eor\`eme~2]{delaBreteche2001}, which gives the exact degree outright, and two of its three hypotheses hold at $\sigma^*$ for every $\w$: one has
\[
H(\mathbf{0}) = \prod_{g \notin T^{\circ}} \zeta\big( \langle g, \sigma^* \rangle \big) \cdot G(\sigma^*) > 0
\]
by \cref{lem:L3}~iv), and the requirement that $\sum_i s_i$ lie in the open cone generated by the polar forms is exactly \cref{eq-strict-comp}, which Goldman--Tucker supplies. The third asks that the polar family have full rank $r = n+1$, equivalently that the dual optimum of \cref{eq-LP} be unique, and this one can fail.

It fails already for the well-formed weights $\w = (2,2,3,3)$, whose dual optima form a segment. At its relative interior point $\sigma^* = (\tfrac{1}{2}, \tfrac{1}{2}, \tfrac{1}{2}, \tfrac{1}{2})$ the tight set is
\[
T^{\circ} = \left\{ (1,1,0,0), \ (0,0,2,0), \ (0,0,0,2), \ (0,0,1,1) \right\} ,
\]
of rank $3 < 4$; at the endpoints of the segment the rank does reach $4$, but there the cone condition fails instead. No point of the segment is admissible, and the exact degree rests on \cref{lem:lower}. On the other side of the dichotomy stand the weights $\w = (1,2,3,5)$ of \cref{prop:1235}, for which $r = n+1$ and \cite[Th\'eor\`eme~2]{delaBreteche2001} would suffice; there, however, the constant is obtained in \cref{thm:diagonal} without any appeal to \cite{delaBreteche2001}. For $(2,2,3,3)$ itself the constant is computed in \cref{thm:2233} by an elementary argument.
\end{rem}

\begin{prop}\label{prop:assembly}
Let $\w = (q_0, \dots, q_n)$ with $\gcd(q_0, \dots, q_n) = 1$, and let $N_{\w}$ be the tuple count of \cref{lem:lower}. Then for every $X \geq 1$,
\[
Z^{\circ}_{\h}\big( \WP^n_{\w}(\Q), X \big) = 2^{n} \, N_{\w}\big( X^{q} \big) .
\]
\end{prop}

\begin{proof}
By \cref{eq-h-pullback}, a point $\p$ with all coordinates nonzero satisfies $\h(\p) \leq X$ if and only if $H(\phi(\p)) \leq B$, where $B = X^q$, and $\phi(\p)$ has all coordinates nonzero. Since every rational preimage of a point with nonzero coordinates has nonzero coordinates, grouping the points of $\WP^n_{\w}(\Q)$ by their images gives
\[
Z^{\circ}_{\h}\big( \WP^n_{\w}(\Q), X \big) = \sum_{y} \# \phi^{-1}(y)(\Q) = m_{\w} \cdot L(B),
\]
where the sum runs over the liftable $y \in \bP^n(\Q)$ with all coordinates nonzero and $H(y) \leq B$, the quantity $L(B)$ is the number of such $y$, and $m_{\w}$ is the fiber multiplicity of \cref{lem:fiber-count}: with $\nu = \#\{ i : n_i \text{ even} \}$, one has $m_{\w} = 2^{\nu - 1}$ if $q$ is even and $m_{\w} = 1$ if $q$ is odd, constant over such $y$.

We count $L(B)$ by primitive integral representatives. Each $y$ has exactly two, $\pm(y_0, \dots, y_n)$, with all $y_i \neq 0$, $\gcd(y_0, \dots, y_n) = 1$, and $H(y) = \max_i |y_i|$. Primitivity says $\min_i v_p(y_i) = 0$ at every prime, so by \cref{prop:lifting-criterion} the point is liftable if and only if the vector $\big( v_p(|y_i|) \big)_i$ lies in $M_{\w}$ for every prime $p$, a condition on the absolute values alone, and the coordinates $y_i$ with $n_i$ even share a sign, a condition preserved by $y \mapsto -y$.

Fix $u \in \Z_{\geq 1}^{n+1}$ with $\max_i u_i \leq B$ and valuation vectors in $M_{\w}$; these $u$ are counted by $N_{\w}(B)$. The primitive tuples with $|y| = u$ are the $2^{n+1}$ sign patterns $\varepsilon \odot u$, $\varepsilon \in \{\pm 1\}^{n+1}$. If $q$ is odd, then $\nu = 0$, the sign condition is vacuous, and all $2^{n+1}$ patterns are liftable, giving $2^{n+1}/2 = 2^n$ liftable points per $u$ and
\[
Z^{\circ}_{\h} = 1 \cdot L(B) = 2^n N_{\w}(B) .
\]
If $q$ is even, then $\nu \geq 1$ by coprimality of the weights, the admissible patterns are those constant on the $\nu$ even-exponent coordinates, in number $2 \cdot 2^{n+1-\nu}$, giving $2^{n+1-\nu}$ liftable points per $u$ and
\[
Z^{\circ}_{\h} = 2^{\nu - 1} \cdot L(B) = 2^{\nu - 1} \cdot 2^{n+1-\nu} N_{\w}(B) = 2^n N_{\w}(B) .
\]
In both cases the identity holds.
\end{proof}

\begin{proof}[Proof of \cref{thm:main}]
Combine \cref{prop:assembly} with \cref{lem:L4} at $B = X^q$:
\[
Z^{\circ}_{\h}\big( \WP^n_{\w}(\Q), X \big) = 2^n \left( X^{q \, a(\w)} \, \Psi\big( q \log X \big) + O\big( X^{q(a(\w) - \theta)} \big) \right),
\]
so the theorem holds with $P_{\w}(x) = 2^n \Psi(qx)$, of the same exact degree $\beta(\w)$ as $\Psi$ and with leading coefficient $c_{\w} = 2^n q^{\beta(\w)} c_{\Psi} > 0$, and with $\theta$ replaced by $q\theta$.
\end{proof}

\begin{rem}\label{rem-constant}
The Tauberian theorem of \cite{delaBreteche2001} produces the polynomial $P_{\w}$, and in particular the leading constant $c_{\w}$, but not in closed form: they are assembled from residues of the multivariable series $F$ along the polar divisors indexed by $T^{\circ}$. The expected shape of $c_{\w}$ is a product of local densities attached to the lifting set $M_{\w}$, namely $H(\mathbf 0) = \prod_p (1 - p^{-1})^{\varrho} \Phi_p(\sigma^*)$ with $\varrho = |T^{\circ}|$, as in \cref{rem-nondegenerate}, against a volume attached to the optimal face of \cref{eq-LP}, as for the conjectural constants of the classical Batyrev--Manin problem. In \cref{sec-constants} this is proved in two cases: when $T^{\circ}$ consists of the axis generators, where the volume factor is $1$ and the constant is exactly this Euler product (\cref{thm:diagonal}), and for $(2,2,3,3)$ (\cref{thm:2233}). An explicit error exponent $\theta$ is obtained in the first case only.
\end{rem}

\section{The leading constant in two cases}\label{sec-constants}

The Tauberian theorem produces $P_{\w}$ from residues of $F$ along the polar divisors indexed by $T^{\circ}$, but not in closed form. In this section we evaluate the leading term in two situations: when the tight set consists of the axis generators alone, where a convolution argument gives the asymptotic with an Euler product constant, and for $\w = (2,2,3,3)$, the degenerate example of \cref{rem-nondegenerate}, where Möbius inversion gives the constant. Throughout, $N_{\w}(B)$ is the tuple count of \cref{lem:lower}, so that $Z^{\circ}_{\h}(\WP^n_{\w}(\Q), X) = 2^n N_{\w}(X^q)$ by \cref{prop:assembly}.

\subsection{The diagonal case}\label{subsec-diagonal}

Recall $\ell_i = \lcm_{j \neq i} \gcd(n_i, n_j)$ and $\sigma_0 = (1/\ell_0, \dots, 1/\ell_n)$ from \cref{prop:a-bounds}.

\begin{thm}\label{thm:diagonal}
Let $n \geq 1$, $\gcd(q_0, \dots, q_n) = 1$, and assume
\begin{equation}\label{eq-diagonal-hyp}
\langle g, \sigma_0 \rangle > 1 \qquad \text{for every } g \in T_{\w} \setminus \{ \ell_0 e_0, \dots, \ell_n e_n \} .
\end{equation}
Then $\sigma_0$ is the unique dual optimum, $a(\w) = \sum_i 1/\ell_i$, $\beta(\w) = 0$, and there is $\theta > 0$ such that
\[
N_{\w}(B) = c^{\flat}_{\w} \, B^{a(\w)} + O\big( B^{a(\w) - \theta} \big), \qquad
c^{\flat}_{\w} = \prod_p \Big( 1 - \frac{1}{p} \Big)^{n+1} \Phi_p(\sigma_0), \qquad
\Phi_p(\sigma_0) = \sum_{t \in M_{\w}} p^{-\langle t, \sigma_0 \rangle} ,
\]
the product converging absolutely. Consequently
\[
Z^{\circ}_{\h}\big( \WP^n_{\w}(\Q), X \big) = 2^n c^{\flat}_{\w} \, X^{q\,a(\w)} + O\big( X^{q\,a(\w) - q\theta} \big) .
\]
One may take for $\theta$ any $\varepsilon > 0$ with $\ell_i \varepsilon < \delta_{\w}$ for all $i$ and $g_i \varepsilon < \langle g, \sigma_0 \rangle - 1$ for all $g \in T_{\w} \setminus T^{\circ}$ and all $i$, where $\delta_{\w}$ is the constant of \cref{lem:L3}. The local factor is an explicit rational function of $p^{-1/L}$, $L = \lcm_i \ell_i$: with $D$, $R_{\w}$ as in \cref{lem:L2}~i),
\begin{equation}\label{eq-local-factor-formula}
\Big( 1 - \frac{1}{p} \Big)^{n+1} \Phi_p(\sigma_0)
= \Big( 1 - \frac{1}{p} \Big)^{n+1} \big( 1 - p^{-\langle \mathbf{1}, \sigma_0 \rangle} \big)
\frac{\sum_{r \in R_{\w}} p^{-\langle r, \sigma_0 \rangle}}{\prod_{i=0}^{n} \big( 1 - p^{-D/\ell_i} \big)} .
\end{equation}
\end{thm}

\begin{proof}
\emph{The linear program.} Every nonzero $t \in M_{\w}$ dominates some $g \in T_{\w}$, and $\langle \cdot, \sigma_0 \rangle$ is monotone, so \cref{eq-diagonal-hyp} together with $\langle \ell_i e_i, \sigma_0 \rangle = 1$ gives $\langle t, \sigma_0 \rangle \geq 1$ for all nonzero $t \in M_{\w}$, with equality only for $t = \ell_i e_i$. Thus $\sigma_0$ is dual feasible with value $\sum_i 1/\ell_i$, every dual feasible $\sigma$ has $\sigma_i \geq 1/\ell_i$ (axis constraints) and hence value $\geq \sum_i 1/\ell_i$, and $\sigma_0$ is the unique optimum. Its tight set is $T^{\circ} = \{ \ell_i e_i \}$, of rank $n+1$, so $\beta(\w) = \varrho - r = 0$ by \cref{eq-face-dim}.

\emph{Factorization as a convolution.} Let $f$ be the indicator of the admissible tuples, so $F(s) = \sum_y f(y) y^{-s}$, and let $\mathbf{1}_{\ell}(v) = 1$ if $v_i$ is an $\ell_i$-th power for every $i$ and $0$ otherwise, so that $\sum_v \mathbf{1}_{\ell}(v) v^{-s} = \prod_i \zeta(\ell_i s_i)$. Put
\[
\widetilde{H}(s) = \prod_{g \in T_{\w} \setminus T^{\circ}} \zeta\big( \langle g, s \rangle \big) \cdot G(s) ,
\]
so that $F = \prod_i \zeta(\ell_i s_i) \cdot \widetilde{H}$ on the region of \cref{lem:L3}~i). Each $\zeta(\langle g, s \rangle) = \sum_{w \geq 1} \prod_i (w^{g_i})^{-s_i}$ is a Dirichlet series in $n+1$ variables with nonnegative coefficients, absolutely convergent when $\langle g, \Re s \rangle > 1$; and $G(s) = \prod_p E_p(s)$ with $E_p(s) = \sum_u a_u p^{-\langle u, s \rangle}$ by \cref{lem:L2}~ii), so $G$ is the Dirichlet series $\sum_y b^G_y y^{-s}$ with $b^G_y = \prod_{p} a_{v_p(y)}$, and
\[
\sum_y |b^G_y| \, y^{-\sigma} = \prod_p \sum_u |a_u| p^{-\langle u, \sigma \rangle} \leq \prod_p \big( 1 + A \, p^{-\mu(\sigma)} \big) < \infty
\]
for $\sigma \in \Omega_{\delta_{\w}}$, by the proof of \cref{lem:L2}~iii) and \cref{lem:L3}~ii). A product of absolutely convergent Dirichlet series is an absolutely convergent Dirichlet series whose coefficients are the convolution of the factors' coefficients; hence $\widetilde{H}(s) = \sum_u b_u u^{-s}$ with $\sum_u |b_u| u^{-\sigma} < \infty$ whenever $\sigma \in \Omega_{\delta_{\w}}$ and $\langle g, \sigma \rangle > 1$ for all $g \in T_{\w} \setminus T^{\circ}$. Both conditions hold at $\sigma_0$, and, by \cref{eq-diagonal-hyp} and openness, at $\sigma_0 - \varepsilon e_i$ for every $i$ and all sufficiently small $\varepsilon > 0$: indeed $\langle \ell_j e_j, \sigma_0 - \varepsilon e_i \rangle = 1 - \delta_{ij} \ell_i \varepsilon > 1 - \delta_{\w}$, and $\langle g, \sigma_0 - \varepsilon e_i \rangle = \langle g, \sigma_0 \rangle - g_i \varepsilon > 1$ for $g \notin T^{\circ}$ once $\varepsilon$ is small. Fix such an $\varepsilon$, with $\varepsilon \leq \min_i 1/\ell_i$ as well, and write $\mathcal{S} = \sum_u |b_u| u^{-\sigma_0} + \sum_i \sum_u |b_u| u^{-(\sigma_0 - \varepsilon e_i)} < \infty$.

Two absolutely convergent Dirichlet series in $n+1$ variables that agree as functions on a nonempty open set have the same coefficients (restrict to the lines $s = \sigma + z \mathbf{1}$, on which both become one-variable Dirichlet series in $z$ with exponents $\log(y_0 \cdots y_n)$ and coefficients $\sum_{y_0 \cdots y_n = N} c_y y^{-\sigma}$, apply the one-variable uniqueness theorem, and vary $\sigma$ over an open set, which separates the $y$ with a common product). Hence $f = \mathbf{1}_\ell * b$, that is,
\[
f(y) = \sum_{u \cdot v = y} \mathbf{1}_\ell(v) \, b_u \qquad \text{for every } y \in \Z^{n+1}_{\geq 1},
\]
where $u \cdot v$ is the coordinatewise product.

\emph{Counting.} Summing over the box $\{ y_i \leq B \}$ and writing $v_i = w_i^{\ell_i}$,
\begin{equation}\label{eq-diag-count}
N_{\w}(B) = \sum_{u} b_u \prod_{i=0}^{n} \Big\lfloor \big( B / u_i \big)^{1/\ell_i} \Big\rfloor ,
\end{equation}
the sum being finite since the product vanishes unless $u_i \leq B$ for all $i$. Write $\lfloor (B/u_i)^{1/\ell_i} \rfloor = (B/u_i)^{1/\ell_i} - \rho_i(u)$ with $0 \leq \rho_i(u) < 1$ when $u_i \leq B$, and note that $\prod_i \lfloor (B/u_i)^{1/\ell_i} \rfloor = 0$ when some $u_i > B$. Then
\[
N_{\w}(B) = B^{a(\w)} \sum_{u} b_u \, u^{-\sigma_0} \;-\; \Sigma_1 \;-\; \Sigma_2 ,
\]
where $\Sigma_1 = \sum_{u_i > B \text{ for some } i} b_u \, B^{a(\w)} u^{-\sigma_0}$ and $\Sigma_2$ collects, over the $u$ with all $u_i \leq B$, the terms of the expansion of $\prod_i \big( (B/u_i)^{1/\ell_i} - \rho_i(u) \big)$ containing at least one factor $\rho_i(u)$. The main term is $B^{a(\w)} \widetilde{H}(\sigma_0)$.

For $\Sigma_1$: if $u_i > B$ then $1 \leq (u_i / B)^{\varepsilon}$, so
\[
|\Sigma_1| \leq B^{a(\w)} \sum_i \sum_{u} |b_u| \, u^{-\sigma_0} \Big( \frac{u_i}{B} \Big)^{\varepsilon} = B^{a(\w) - \varepsilon} \sum_i \sum_u |b_u| \, u^{-(\sigma_0 - \varepsilon e_i)} \leq \mathcal{S} \, B^{a(\w) - \varepsilon} .
\]
For $\Sigma_2$: a term containing $\rho_i(u)$ is a product of some of the factors $(B/u_j)^{1/\ell_j}$, $j \neq i$, and of some $\rho_j(u)$, $|\rho_j| < 1$; since every $(B/u_j)^{1/\ell_j} \geq 1$ when $u_j \leq B$, its absolute value is at most
\[
\prod_{j \neq i} \Big( \frac{B}{u_j} \Big)^{1/\ell_j} = B^{a(\w)} u^{-\sigma_0} \Big( \frac{u_i}{B} \Big)^{1/\ell_i} \leq B^{a(\w)} u^{-\sigma_0} \Big( \frac{u_i}{B} \Big)^{\varepsilon} ,
\]
the last step because $u_i / B \leq 1$ and $\varepsilon \leq 1/\ell_i$. There are fewer than $2^{n+1}$ such terms for each $u$, so
\[
|\Sigma_2| \leq 2^{n+1} \sum_i B^{a(\w) - \varepsilon} \sum_u |b_u| \, u^{-(\sigma_0 - \varepsilon e_i)} \leq 2^{n+1} \mathcal{S} \, B^{a(\w) - \varepsilon} .
\]
Altogether
\[
N_{\w}(B) = \widetilde{H}(\sigma_0) \, B^{a(\w)} + O\big( B^{a(\w) - \varepsilon} \big) ,
\]
which is the asymptotic with $\theta = \varepsilon$.

\emph{The constant.} By the local factorization of \cref{lem:L2}~ii), for every $\sigma$ in the region of absolute convergence,
\[
\widetilde{H}(\sigma) = \frac{F(\sigma)}{\prod_i \zeta(\ell_i \sigma_i)} = \prod_p \Phi_p(\sigma) \prod_i \big( 1 - p^{-\ell_i \sigma_i} \big) ,
\]
and both sides extend continuously to $\sigma_0$, the right side because $\prod_p \Phi_p(\sigma) \prod_{g \in T_{\w}} (1 - p^{-\langle g, \sigma \rangle}) = G(\sigma)$ converges uniformly on $\Omega_{\delta_{\w}}$ and the remaining factors $\prod_{g \notin T^{\circ}} \zeta(\langle g, \sigma \rangle)$ are continuous near $\sigma_0$. At $\sigma_0$ every $\ell_i \sigma_{0,i} = 1$, giving $c^{\flat}_{\w} = \prod_p (1 - 1/p)^{n+1} \Phi_p(\sigma_0)$, positive since every factor is (\cref{lem:L2}~iii)). Formula \cref{eq-local-factor-formula} is \cref{lem:L2}~i) evaluated at $z_i = p^{-1/\ell_i}$, where $z_i^D = p^{-D/\ell_i}$ and $z^{\mathbf 1} = p^{-\langle \mathbf{1}, \sigma_0 \rangle}$. The statement for $Z^{\circ}_{\h}$ is \cref{prop:assembly}.
\end{proof}

\begin{rem}
The hypothesis \cref{eq-diagonal-hyp} is decidable by inspecting the finite set $T_{\w}$, which lies in $[0, D]^{n+1}$: if some $g \in T_{\w}$ had a coordinate $g_i \geq D$ then $g - D e_i$ would be a nonzero element of $M_{\w}$ strictly below $g$, unless $g = D e_i$, in which case $D = \ell_i$. The pairwise coprime case is the trivial instance, $\ell_i = 1$, $T_{\w} = T^{\circ} = \{e_i\}$, $\Phi_p(\mathbf{1}) = \sum_{t \geq 0, \min t = 0} p^{-\langle t, \mathbf{1} \rangle} = (1-p^{-1})^{-(n+1)} (1 - p^{-(n+1)})$, so $c^{\flat}_{\w} = 1/\zeta(n+1)$ and \cref{thm:diagonal} is Schanuel's theorem for the coordinate-nonzero points of $\bP^n(\Q)$ transported by \cref{prop:assembly}. The first nontrivial instance is $\w = (1,2,3,5)$, treated in \cref{prop:1235}. It fails for $(1,2,3)$, where $q\,a(\w) = Q$ but non-axis generators are tight, and for $(2,3,5)$.
\end{rem}

\subsection{The weights $(2,2,3,3)$}\label{subsec-2233}

For $\w = (2,2,3,3)$ one has $q = 6$, $Q = 10$, $(n_0, n_1, n_2, n_3) = (3,3,2,2)$, $\gcd(n_0, n_1) = 3$, $\gcd(n_2, n_3) = 2$, and all other pairs are coprime. Hence
\[
M_{\w} = \big\{ t \in \Z^4_{\geq 0} : \min_i t_i = 0, \ t_0 \equiv t_1 \ (3), \ t_2 \equiv t_3 \ (2) \big\},
\]
and a primitive tuple $y = (y_0, y_1, y_2, y_3) \in \Z^4_{\geq 1}$ is admissible if and only if
\begin{equation}\label{eq-2233-adm}
\frac{y_1}{y_0} \in (\Q^\times)^3, \qquad \frac{y_3}{y_2} \in (\Q^\times)^2, \qquad \gcd(y_0, y_1, y_2, y_3) = 1 .
\end{equation}
By \cref{subsec-lp-examples} the linear program has $a(\w) = 2$ and $\beta(\w) = 1$.

\begin{thm}\label{thm:2233}
For $\w = (2,2,3,3)$,
\[
N_{\w}(B) = \Big( \frac{12}{\pi^2 \zeta(3)} - \frac{36}{\pi^4} \Big) B^2 \log B + O(B^2), \qquad
Z^{\circ}_{\h}\big( \WP^3_{(2,2,3,3)}(\Q), X \big) = 48 \Big( \frac{12}{\pi^2 \zeta(3)} - \frac{36}{\pi^4} \Big) X^{12} \log X + O(X^{12}) .
\]
Numerically the first constant is $0.64190\ldots$ and the second $30.811\ldots$.
\end{thm}

\begin{proof}
For $x \geq 1$ let
\[
A(x) = \#\big\{ (a,b) \in \Z^2_{\geq 1} : a, b \leq x, \ b/a \in (\Q^\times)^3 \big\}, \qquad
\mathcal{B}(x) = \#\big\{ (c,e) \in \Z^2_{\geq 1} : c, e \leq x, \ e/c \in (\Q^\times)^2 \big\} .
\]
Since $b/a$ is a rational cube if and only if $a$ and $b$ have the same cube-free kernel $k$, writing $a = k j^3$ and $b = k j'^3$ gives
\[
A(x) = \sum_{\substack{k \leq x \\ k \text{ cube-free}}} \big\lfloor (x/k)^{1/3} \big\rfloor^2
= \sum_{m \geq 1} (2m - 1) \, \#\{ k \text{ cube-free} : k \leq x/m^3 \} ,
\]
the second form by writing $\lfloor u \rfloor^2 = \sum_{m \leq u} (2m-1)$ and interchanging the sums. The number of cube-free integers up to $y$ is $y/\zeta(3) + O(y^{1/3})$, so
\[
A(x) = \frac{x}{\zeta(3)} \sum_{m \leq x^{1/3}} \frac{2m-1}{m^3} + O\Big( x^{1/3} \sum_{m \leq x^{1/3}} \frac{2m-1}{m} \Big)
= \frac{2\zeta(2) - \zeta(3)}{\zeta(3)} \, x + O\big( x^{2/3} \big) ,
\]
using $\sum_{m > x^{1/3}} (2m-1)/m^3 = O(x^{-1/3})$. In the same way, with the count $6y/\pi^2 + O(\sqrt y)$ of square-free integers up to $y$,
\[
\mathcal{B}(x) = \sum_{m \geq 1} (2m-1) \, \#\{ k \text{ square-free} : k \leq x/m^2 \}
= \frac{6x}{\pi^2} \sum_{m \leq \sqrt x} \frac{2m-1}{m^2} + O\Big( \sqrt{x} \sum_{m \leq \sqrt x} \frac{2m-1}{m} \Big)
= \frac{6}{\pi^2} \, x \log x + O(x) ,
\]
since $\sum_{m \leq \sqrt x} (2m-1)/m^2 = \log x + O(1)$.

Now $N_{\w}(B)$ counts the quadruples satisfying \cref{eq-2233-adm} with all coordinates $\leq B$. The first two conditions of \cref{eq-2233-adm} are invariant under $y \mapsto d \cdot y$ for any integer $d \geq 1$, since multiplying both members of a ratio by $d$ does not change it, and the third is detected by Möbius inversion:
\[
N_{\w}(B) = \sum_{d \leq B} \mu(d) \, A(B/d) \, \mathcal{B}(B/d) .
\]
Insert the asymptotics. Writing $c_A = (2\zeta(2) - \zeta(3))/\zeta(3)$,
\[
A(B/d) \mathcal{B}(B/d) = c_A \frac{6}{\pi^2} \Big( \frac{B}{d} \Big)^2 \log \frac{B}{d} + O\Big( \Big( \frac{B}{d} \Big)^2 \Big) + O\Big( \Big( \frac{B}{d} \Big)^{5/3} \log \frac{B}{d} \Big) ,
\]
and summing over $d$ with $\sum_d d^{-2} < \infty$, $\sum_d d^{-5/3} < \infty$,
\[
N_{\w}(B) = c_A \frac{6}{\pi^2} B^2 \Big( \log B \sum_{d \leq B} \frac{\mu(d)}{d^2} - \sum_{d \leq B} \frac{\mu(d) \log d}{d^2} \Big) + O(B^2)
= c_A \frac{6}{\pi^2} \cdot \frac{6}{\pi^2} \, B^2 \log B + O(B^2) ,
\]
because $\sum_d \mu(d) d^{-2} = 1/\zeta(2) = 6/\pi^2$, the tails are $O(1/B)$, and $\sum_d \mu(d) \log d / d^2$ converges. Since $2\zeta(2) = \pi^2/3$,
\[
c_A \frac{36}{\pi^4} = \frac{36}{\pi^4} \Big( \frac{\pi^2}{3 \zeta(3)} - 1 \Big) = \frac{12}{\pi^2 \zeta(3)} - \frac{36}{\pi^4} .
\]
Finally $Z^{\circ}_{\h}(X) = 2^3 N_{\w}(X^6)$ by \cref{prop:assembly}, and $\log X^6 = 6 \log X$ gives the factor $48$.
\end{proof}

\begin{rem}
The proof gives $N_{\w}(B) = B^2 ( c \log B + c' ) + o(B^2)$ with an explicit $c'$ involving $\zeta'(2)$ and the constant term of $\mathcal{B}(x)/x$, and the data of \cref{subsec-2233-data} are compared with this two-term form. The same method applies to every $\w$ whose Kummer graph $\Gamma_{\w}$ of \cref{lem:L1} is a disjoint union of cliques on which the congruence moduli are constant, since admissibility then factors as in \cref{eq-2233-adm}. For $(2,3,5)$ the graph is a triangle with three different moduli and no such factorization exists.
\end{rem}

\section{Examples and numerical verification}\label{sec-numerics}

We work out the linear program \cref{eq-LP} for four weight vectors and compare \cref{thm:main} with exact values of $N_{\w}(B)$. The sets $T_{\w}$ were obtained by enumerating $M_{\w} \cap [0, D]^{n+1}$, which contains $T_{\w}$ by the remark following \cref{thm:diagonal}; the optima are verified below by exhibiting primal and dual solutions. The counts were obtained by enumerating admissible tuples over the primes $p \leq B^{1/m_{\w}}$, where $m_{\w} = \min_{g \in T_{\w}} \max_i g_i$, and agree with a direct enumeration of all tuples for small $B$.

\subsection{The linear programs}\label{subsec-lp-examples}

We use the dual program of \cref{prop:a-bounds}: a pair $(c, \sigma)$ with $c$ primal feasible, $\sigma$ dual feasible and $\sum_g c_g = \sum_i \sigma_i$ is optimal on both sides, and by complementary slackness every optimal $c$ is supported on the tight set $T^{\circ}(\sigma) = \{ g : \langle g, \sigma \rangle = 1 \}$, so that, when $\sigma_i > 0$ for all $i$, the primal optimal face is $\Phi = \{ c \geq 0 : \operatorname{supp}(c) \subseteq T^{\circ}(\sigma), \ \sum_g c_g g = \mathbf{1} \}$ and $\beta(\w) = \dim \Phi$.

\subsubsection{$\w = (2,3,5)$}\label{subsec-235}
Here $q = 30$, $Q = 10$, $(n_0, n_1, n_2) = (15, 10, 6)$, $g_{01} = 5$, $g_{02} = 3$, $g_{12} = 2$, $D = 30$, $(\ell_0, \ell_1, \ell_2) = (15, 10, 6)$, and
\[
M_{\w} = \{ t \in \Z^3_{\geq 0} : \min t = 0, \ t_0 \equiv t_1 \ (5), \ t_0 \equiv t_2 \ (3), \ t_1 \equiv t_2 \ (2) \} .
\]
The minimal elements are
\[
T_{\w} = \{ (15,0,0), \ (0,10,0), \ (0,0,6), \ (5,0,2), \ (0,5,3), \ (3,8,0), \ (6,6,0), \ (9,4,0), \ (12,2,0) \} ,
\]
nine in all. Take $\sigma^* = (2/15, 1/10, 1/6)$. Its values on $T_{\w}$ are $2, 1, 1, 1, 1, 6/5, 7/5, 8/5, 9/5$, so $\sigma^*$ is dual feasible with $T^{\circ} = \{ (0,10,0), (0,0,6), (5,0,2), (0,5,3) \}$ and value $2/15 + 1/10 + 1/6 = 2/5$. On the primal side, $c = \tfrac15 \big( [(5,0,2)] + [(0,5,3)] \big)$ is feasible: $\sum_g c_g g = (1, 1, 1)$, with value $2/5$. Hence
\[
a(\w) = \tfrac{2}{5}, \qquad q\,a(\w) = 12 > Q = 10 .
\]
The primal optimal face is $\{ c_1 (5,0,2) + c_2 (0,5,3) + c_3 (0,0,6) + c_4 (0,10,0) = \mathbf{1}, \ c \geq 0 \}$, that is $c_1 = 1/5$ and $(c_2, c_3, c_4) = \big( (1 - 10u)/5, \, u, \, u \big)$, $0 \leq u \leq 1/10$, a segment; so $\beta(\w) = 1$, in accordance with $\varrho - r = 4 - 3$. The dual optimum is unique, since $c_3, c_4 > 0$ in the relative interior force $\langle (0,0,6), \sigma \rangle = \langle (0,10,0), \sigma \rangle = 1$ and then the two remaining tight constraints determine $\sigma^*$. The failure of \cref{eq-Q-criterion} is witnessed by $(5,0,2)$: $5 \cdot 2 + 2 \cdot 5 = 20 < 30$. \cref{thm:main} gives
\[
Z^{\circ}_{\h}\big( \WP^2_{(2,3,5)}(\Q), X \big) \sim c_{\w} \, X^{12} \log X ,
\]
while $Z^{\circ}_{\hst} \asymp X^{10}$ and $Z^{\circ}_H(\bP^2(\Q), X^{30}) \asymp X^{90}$. The polar family has full rank $r = 3 = n+1$ and $c^*$ is strictly positive on $T^{\circ}$, so this is a case in which \cite[Th\'eor\`eme~2]{delaBreteche2001} would also give the exact degree.

\subsubsection{$\w = (2,2,3,3)$}
Here $q = 6$, $Q = 10$, $(n_i) = (3,3,2,2)$, $D = 6$, $(\ell_i) = (3,3,2,2)$, and
\[
T_{\w} = \{ (3,0,0,0), \ (0,3,0,0), \ (0,0,2,0), \ (0,0,0,2), \ (1,1,0,0), \ (0,0,1,1) \} .
\]
The dual constraints are $\sigma_0, \sigma_1 \geq 1/3$, $\sigma_2, \sigma_3 \geq 1/2$, $\sigma_0 + \sigma_1 \geq 1$, $\sigma_2 + \sigma_3 \geq 1$, so the dual value is $2$, attained exactly on the segment $\{ \sigma_0 + \sigma_1 = 1, \ 1/3 \leq \sigma_0 \leq 2/3, \ \sigma_2 = \sigma_3 = 1/2 \}$; the primal point $c = [(1,1,0,0)] + [(0,0,1,1)]$ has value $2$. Thus $a(\w) = 2$, $q\,a(\w) = 12 > Q$. At the relative interior point $\sigma^* = (\tfrac12, \tfrac12, \tfrac12, \tfrac12)$ the tight set is $\{ (1,1,0,0), (0,0,2,0), (0,0,0,2), (0,0,1,1) \}$, of rank $3$, and the primal optimal face $\{ c_1 = 1, \ 2c_2 + c_4 = 1, \ 2 c_3 + c_4 = 1, \ c \geq 0 \}$ is a segment, $\beta(\w) = 1$. This is the degenerate case of \cref{rem-nondegenerate}, and \cref{thm:2233} gives the constant.

\subsubsection{$\w = (1,2,3,5)$}\label{subsec-1235}
Here $q = 30$, $Q = 11$, $(n_i) = (30, 15, 10, 6)$, $D = 30$, $\ell_i = n_i$ for all $i$, and the congruences reduce to $t_j \equiv t_0 \pmod{n_j}$ for $j = 1,2,3$. The set $T_{\w}$ has eleven elements,
\[
\begin{gathered}
(30,0,0,0), \ (0,15,0,0), \ (0,0,10,0), \ (0,0,0,6), \\
(6,6,6,0), \ (15,0,5,3), \ (10,10,0,4), \ (12,12,2,0), \ (20,5,0,2), \ (18,3,8,0), \ (24,9,4,0) .
\end{gathered}
\]
Against $\sigma_0 = \w/30 = (1/30, 1/15, 1/10, 1/6)$ the non-axis elements take the values $6/5$, $3/2$, $5/3$, $7/5$, $4/3$, $8/5$, $9/5$, all $> 1$, so the hypothesis \cref{eq-diagonal-hyp} of \cref{thm:diagonal} holds, $\sigma_0$ is the unique dual optimum, $a(\w) = 11/30$, $q\,a(\w) = 11 = Q$, $\beta(\w) = 0$. Criterion \cref{eq-Q-criterion} can also be checked by hand: if $t \in M_{\w}$ is nonzero with $t_0 = 0$ then $n_j \mid t_j$ for $j \geq 1$ and $\langle t, \w \rangle \geq \min_j q_j n_j = 30$; if $t_0 = a > 0$ and $t_j = 0$ then $n_j \mid a$, and the congruences $t_i \equiv a \pmod{n_i}$ for the remaining $i$ force $\langle t, \w \rangle \geq 30$ in each of the finitely many cases $a \in \{6, 10, 12, 15, 18, 20, 24, 30\}$, the smallest value $30$ being attained only by the axis vectors.

\subsubsection{$\w = (1,2,3)$}\label{subsec-123}
Here $q = 6 = Q$, $(n_i) = (6,3,2)$, $D = 6$, $(\ell_i) = (6,3,2)$, and
\[
T_{\w} = \{ (6,0,0), \ (0,3,0), \ (0,0,2), \ (2,2,0), \ (3,0,1), \ (4,1,0) \} .
\]
Every element has $\langle g, \sigma_0 \rangle = 1$ for $\sigma_0 = (1/6, 1/3, 1/2)$: $2/6 + 2/3 = 3/6 + 1/2 = 4/6 + 1/3 = 1$. Hence $\sigma_0$ is the unique dual optimum, $a(\w) = 1$, $q\,a(\w) = 6 = Q$, $T^{\circ} = T_{\w}$ has six elements of rank $3$, and $\beta(\w) = 3$: the primal optimal face $\{ c \in \R^6_{\geq 0} : \sum_g c_g g = \mathbf{1} \}$ is the three-dimensional polytope
\[
c_{(6,0,0)} = \tfrac16 \big( 1 - 2 c_{(2,2,0)} - 3 c_{(3,0,1)} - 4 c_{(4,1,0)} \big), \quad
c_{(0,3,0)} = \tfrac13 \big( 1 - 2 c_{(2,2,0)} - c_{(4,1,0)} \big), \quad
c_{(0,0,2)} = \tfrac12 \big( 1 - c_{(3,0,1)} \big),
\]
with $(c_{(2,2,0)}, c_{(3,0,1)}, c_{(4,1,0)})$ ranging over the simplex $2 c_{(2,2,0)} + 3 c_{(3,0,1)} + 4 c_{(4,1,0)} \leq 1$. Criterion \cref{eq-Q-criterion} holds, so the weighted and tautological heights share the exponent $6$, but \cref{thm:main} gives
\[
Z^{\circ}_{\h}\big( \WP^2_{(1,2,3)}(\Q), X \big) \sim c_{\w} \, X^{6} (\log X)^{3} \qquad \text{against} \qquad Z^{\circ}_{\hst} \sim \frac{4}{\zeta(6)} X^6 .
\]
In terms of the chart $[1 : u : v]$, $u, v \in \Q^\times$, this says that the affine height $H_{\mathrm{aff}}(u^3, v^2)$ is $\leq B$ for $\asymp B (\log B)^3$ pairs $(u,v)$, against $\asymp B^3$ pairs for $H_{\mathrm{aff}}(u,v)$.

\subsection{The constant for $(1,2,3,5)$}

\begin{prop}\label{prop:1235}
For $\w = (1,2,3,5)$ and $u = p^{-1/30}$, the local factor of \cref{thm:diagonal} is the polynomial
\[
\Big( 1 - \frac{1}{p} \Big)^{4} \Phi_p(\sigma_0)
= 1 + u^{36} + u^{40} + u^{42} + u^{45} + u^{48} + u^{50} + u^{54} - u^{66} - u^{70} - u^{72} - u^{75} - u^{78} - u^{80} - u^{84} - u^{120} ,
\]
equivalently
\[
1 + p^{-6/5} + p^{-4/3} + p^{-7/5} + p^{-3/2} + p^{-8/5} + p^{-5/3} + p^{-9/5} - p^{-11/5} - p^{-7/3} - p^{-12/5} - p^{-5/2} - p^{-13/5} - p^{-8/3} - p^{-14/5} - p^{-4} .
\]
Consequently
\[
Z^{\circ}_{\h}\big( \WP^3_{(1,2,3,5)}(\Q), X \big) = 8 \, c^{\flat}_{(1,2,3,5)} \, X^{11} + O\big( X^{11 - \theta} \big), \qquad
c^{\flat}_{(1,2,3,5)} = \prod_p \Big( 1 - \frac{1}{p} \Big)^{4} \Phi_p(\sigma_0) = 35.4 \ldots ,
\]
and by \cref{lem:stratum-height} the same asymptotic holds for $Z^{\circ}_{\h}(\WP^3_{(2,4,6,10)}(\Q), X^{1/2})$.
\end{prop}

\begin{proof}
We evaluate \cref{eq-local-factor-formula}. Here $D = 30$ and $D/\ell_i = q_i$, so the denominator is $(1 - p^{-1})(1 - p^{-2})(1 - p^{-3})(1 - p^{-5})$ and $1 - p^{-\langle \mathbf{1}, \sigma_0 \rangle} = 1 - u^{11}$. The residue set is $R_{\w} = \{ r \in [0, 29]^4 : r_j \equiv r_0 \ (n_j), \ j = 1,2,3 \}$: for each $r_0 = a \in [0,29]$ the coordinate $r_1$ runs over $\{ a \bmod 15, \ (a \bmod 15) + 15 \}$, $r_2$ over $(a \bmod 10) + \{0, 10, 20\}$, and $r_3$ over $(a \bmod 6) + \{0, 6, \dots, 24\}$, so $|R_{\w}| = 30 \cdot 2 \cdot 3 \cdot 5 = 900$. Since $\langle r, \sigma_0 \rangle = (r_0 + 2 r_1 + 3 r_2 + 5 r_3)/30$, the sum over $R_{\w}$ factors as
\[
\sum_{r \in R_{\w}} u^{r_0 + 2r_1 + 3r_2 + 5r_3}
= (1 + u^{30}) (1 + u^{30} + u^{60}) (1 + u^{30} + \cdots + u^{120}) \sum_{a=0}^{29} u^{e(a)},
\qquad e(a) = a + 2 (a \bmod 15) + 3 (a \bmod 10) + 5 (a \bmod 6) .
\]
The three geometric factors are $(1-p^{-2})/(1-p^{-1})$, $(1-p^{-3})/(1-p^{-1})$, $(1-p^{-5})/(1-p^{-1})$, and they cancel against the denominator up to $(1-p^{-1})^{-3}$, which cancels against $(1-p^{-1})^{4}$ up to $(1 - p^{-1}) = (1 - u^{30})$; this last factor cancels against the $(1-p^{-1})^{-1}$ left in the denominator. Altogether
\[
\Big( 1 - \frac{1}{p} \Big)^{4} \Phi_p(\sigma_0) = (1 - u^{11}) \sum_{a=0}^{29} u^{e(a)} .
\]
The exponents $e(a)$, $a = 0, \dots, 29$, are
\[
0, 11, 22, 33, 44, 55, 36, 47, 58, 69, 50, 61, 42, 53, 64, 45, 56, 67, 48, 59, 40, 51, 62, 73, 54, 65, 76, 87, 98, 109 ,
\]
each congruent to $11 a$ modulo $30$. Multiplying by $1 - u^{11}$, the exponents $11, 22, 33, 44, 47, 51, 53, 55, 56, 58, 59, 61, 62, 64, 65, 67, 69, 73, 76, 87, 98, 109$ cancel, leaving the positive terms $0, 36, 40, 42, 45, 48, 50, 54$ and the negative terms $66, 70, 72, 75, 78, 80, 84, 120$. This is the displayed polynomial, whose expansion was also checked symbolically.

The product converges absolutely since the factor is $1 + O(p^{-6/5})$. Its numerical value was computed over $p < 10^7$, the omitted tail $\prod_{p \geq 10^7}$ being $1 + O(10^{-7/5})$ in the logarithm, and the statement for $(2,4,6,10)$ is the reduction to $\w_S / d$ described before \cref{cor:full-count} with $d = 2$.
\end{proof}

\begin{rem}
The exponents in \cref{prop:1235} can be read off $T_{\w}$. The count $N_{\w}(B)$ is dominated by the tuples $(a^{30}, b^{15}, c^{10}, d^{6})$ with $a, b, c, d$ pairwise coprime, which alone contribute $\asymp B^{11/30}$ with the constant $\prod_p(1 - \text{coprimality defect})$; the terms $u^{36}, u^{40}, \dots$ are the contributions of the non-axis generators $(6,6,6,0)$, $(20,5,0,2)$, $\dots$ of $T_{\w}$, whose values at $\sigma_0$ are $6/5, 4/3, \dots$, and the negative terms are the corrections for their overlaps with the axis generators. Because the smallest of these values is $6/5$, the corresponding partial sums $\sum_{w \leq W} w^{-6/5}$ converge to $\zeta(6/5) = 5.59\ldots$ only like $W^{-1/5}$, and the error exponent $\theta$ of \cref{thm:diagonal} is correspondingly small; this is visible in the data of \cref{subsec-1235-data}, where $N_{\w}(B)/B^{11/30}$ is still far from its limit at $B = 10^{11}$.
\end{rem}

\subsection{Exact counts}

\subsubsection{$\w = (2,3,5)$}\label{subsec-235-data}
Here $m_{\w} = 5$, so the enumeration involves only the primes $p \leq B^{1/5}$, and $N_{\w}(B)$ was computed exactly up to $B = 10^{13}$. \cref{thm:main} predicts $N_{\w}(B) = B^{2/5} \big( c_\Psi \log B + c'_\Psi \big) + O(B^{2/5 - \theta})$ with $c_\Psi > 0$.

\begin{table}[h]
\centering
\begin{tabular}{r r r r}
\toprule
$B$ & $N_{\w}(B)$ & $N_{\w}(B)/B^{2/5}$ & $N_{\w}(B)/(B^{2/5} \log B)$ \\
\midrule
$10^{2}$ & $5$ & $0.792$ & $0.172$ \\
$10^{3}$ & $20$ & $1.262$ & $0.183$ \\
$10^{4}$ & $67$ & $1.683$ & $0.183$ \\
$10^{5}$ & $254$ & $2.540$ & $0.221$ \\
$10^{6}$ & $796$ & $3.169$ & $0.229$ \\
$10^{7}$ & $2541$ & $4.027$ & $0.250$ \\
$10^{8}$ & $7635$ & $4.817$ & $0.262$ \\
$10^{9}$ & $23571$ & $5.921$ & $0.286$ \\
$10^{10}$ & $68136$ & $6.814$ & $0.296$ \\
$10^{11}$ & $197821$ & $7.875$ & $0.311$ \\
$10^{12}$ & $566220$ & $8.974$ & $0.325$ \\
$10^{13}$ & $1605900$ & $10.133$ & $0.339$ \\
\bottomrule
\end{tabular}
\caption{Admissible primitive tuples for $\w = (2,3,5)$.}
\label{tab:235}
\end{table}

The third column grows linearly in $\log B$: a least-squares fit over $10^{8} \leq B \leq 10^{13}$ gives $N_{\w}(B)/B^{2/5} \approx 0.457 \log B - 3.62$, with successive differences per decade $1.10, 0.89, 1.06, 1.10, 1.16$ against $0.457 \cdot \log 10 = 1.05$. This is the behavior of a polynomial of degree $1$ in $\log B$ with positive leading coefficient. The fit gives $c_\Psi \approx 0.46$ and $c_{\w} = 2^2 \cdot 30 \cdot c_\Psi \approx 55$ in $Z^{\circ}_{\h} \sim c_{\w} X^{12} \log X$; we have no closed form for these constants.

\subsubsection{$\w = (2,2,3,3)$}\label{subsec-2233-data}
Here $m_{\w} = 1$ and direct enumeration is limited to $B \leq 30$; instead $N_{\w}(B) = \sum_d \mu(d) A(B/d) \mathcal{B}(B/d)$ from the proof of \cref{thm:2233} was evaluated exactly, the functions $A$ and $\mathcal{B}$ being computed from the cube-free and square-free kernels of the integers up to $B$ by a sieve, and the two methods agree for $B = 12, 20, 30$ (values $235$, $727$, $2049$). The two-term prediction is $N_{\w}(B)/(B^2 \log B) \approx c + c'/\log B$ with $c = 12/(\pi^2 \zeta(3)) - 36/\pi^4 = 0.6419$ and $c' = c_A \cdot \frac{6}{\pi^2} \big( -\sum_d \mu(d) \log d / d^2 + c_{\mathcal{B}} \big) \approx 0.51$, where $c_{\mathcal{B}} \approx 0.14$ is the constant term in $\mathcal{B}(x) = \frac{6}{\pi^2} x \log x + c_{\mathcal{B}} x + o(x)$, determined numerically.

\begin{table}[h]
\centering
\begin{tabular}{r r r r}
\toprule
$B$ & $N_{\w}(B)$ & $N_{\w}(B)/(B^2 \log B)$ & $c + c'/\log B$ \\
\midrule
$10^{1}$ & $169$ & $0.7340$ & $0.865$ \\
$10^{2}$ & $30101$ & $0.6536$ & $0.753$ \\
$10^{3}$ & $4615779$ & $0.6682$ & $0.716$ \\
$10^{4}$ & $621400771$ & $0.6747$ & $0.698$ \\
$10^{5}$ & $77888083149$ & $0.6765$ & $0.687$ \\
$10^{6}$ & $9319068551927$ & $0.6745$ & $0.679$ \\
$10^{7}$ & $1082373377988147$ & $0.6715$ & $0.674$ \\
\bottomrule
\end{tabular}
\caption{Admissible primitive tuples for $\w = (2,2,3,3)$; the last column is the two-term prediction of \cref{thm:2233}.}
\label{tab:2233}
\end{table}

The ratio decreases towards $0.642$ from $B = 10^5$ on, and differs from the two-term prediction by $0.4\%$ at $B = 10^7$; the discrepancy at smaller $B$ comes from the $O(x^{2/3})$ term of $A(x)$, whose relative size at $x = 10^{5}$ is about $2\%$.

\subsubsection{$\w = (1,2,3)$}
Here too $m_{\w} = 1$, and $N_{\w}(B) = \sum_d \mu(d) \, C(B/d)$ with $C(x) = \#\{ (m, y_1, y_2) \leq x : y_1/m \in (\Q^\times)^3, \ y_2/m \in (\Q^\times)^2 \} = \sum_{m \leq x} \lfloor (x/k_3(m))^{1/3} \rfloor \lfloor (x/k_2(m))^{1/2} \rfloor$, $k_3, k_2$ the cube-free and square-free kernels; the formula agrees with direct enumeration for $B \leq 100$. \cref{thm:main} predicts $N_{\w}(B) = B \, \Psi(\log B) + O(B^{1-\theta})$ with $\deg \Psi = 3$.

\begin{table}[h]
\centering
\begin{tabular}{r r r r}
\toprule
$B$ & $N_{\w}(B)$ & $N_{\w}(B)/B$ & $\Delta$ per decade \\
\midrule
$10^{1}$ & $12$ & $1.200$ & \\
$10^{2}$ & $188$ & $1.880$ & $0.68$ \\
$10^{3}$ & $3101$ & $3.101$ & $1.22$ \\
$10^{4}$ & $44496$ & $4.450$ & $1.35$ \\
$10^{5}$ & $621384$ & $6.214$ & $1.76$ \\
$10^{6}$ & $8431943$ & $8.432$ & $2.22$ \\
$10^{7}$ & $108479030$ & $10.848$ & $2.42$ \\
\bottomrule
\end{tabular}
\caption{Admissible primitive tuples for $\w = (1,2,3)$.}
\label{tab:123}
\end{table}

The increments of $N_{\w}(B)/B$ per decade increase, so $\Psi$ has degree at least $2$; the leading coefficient of a cubic cannot be isolated from seven data points.

\subsubsection{$\w = (1,2,3,5)$}\label{subsec-1235-data}
Here $m_{\w} = 6$ and the enumeration reaches $B = 10^{11}$ quickly. \cref{prop:1235} gives $N_{\w}(B) = c^{\flat}_{\w} B^{11/30} + O(B^{11/30 - \theta})$ with $c^{\flat}_{\w} = 35.4$.

\begin{table}[h]
\centering
\begin{tabular}{r r r}
\toprule
$B$ & $N_{\w}(B)$ & $N_{\w}(B)/B^{11/30}$ \\
\midrule
$10^{3}$ & $7$ & $0.556$ \\
$10^{4}$ & $19$ & $0.649$ \\
$10^{5}$ & $81$ & $1.189$ \\
$10^{6}$ & $202$ & $1.275$ \\
$10^{7}$ & $560$ & $1.519$ \\
$10^{8}$ & $1592$ & $1.856$ \\
$10^{9}$ & $4221$ & $2.116$ \\
$10^{10}$ & $12346$ & $2.660$ \\
$10^{11}$ & $32891$ & $3.046$ \\
\bottomrule
\end{tabular}
\caption{Admissible primitive tuples for $\w = (1,2,3,5)$.}
\label{tab:1235}
\end{table}

The ratio is far from its limit $35.4$ at $B = 10^{11}$: the constant is $\prod_p (1 - 1/p)^4 \Phi_p(\sigma_0)$, but at finite $B$ only the primes $p \leq B^{1/6}$ and only the generators $g$ with $p^{\max_i g_i} \leq B$ contribute, so the partial Euler product effectively includes the factor $\zeta(6/5)$ of the generator $(6,6,6,0)$ only up to $W \approx B^{1/6}$, and $\sum_{w \leq W} w^{-6/5} = \zeta(6/5) - 5 W^{-1/5} + O(W^{-6/5})$ is at $W = 68$ still $2.6$ below its limit $5.59$. The same applies to the other seven non-axis generators, with exponents $3/2, 5/3, 7/5, 4/3, 8/5, 9/5$. The data therefore test the exponent $11/30$ and not the constant; the exponent $\theta$ furnished by the proof of \cref{thm:diagonal} is less than $(6/5 - 1)/6 = 1/30$ here.

\section{A weighted Batyrev--Manin conjecture}\label{sec-7}

The tautological height counts at the exponent $mQ$ over every number field of degree $m$ (\cref{eq-deng}; \cite{deng1998, bruin2023}). For presentations with pairwise coprime exponents, where $\phi$ is an isomorphism of varieties, the weighted height counts at the exponent $qm(n+1)$ (\cref{prop-hcount} and \cref{rem-chart-k}). For arbitrary coprime weights over $\Q$, \cref{thm:main} gives the interior count $Z^{\circ}_{\h}$ at the exponent $q\,a(\w) \in [Q, q(n+1)]$ with logarithmic power $\beta(\w)$, both read off the linear program \cref{eq-LP}, and \cref{cor:full-count} assembles the full counting function from the coordinate strata.

Both heights depend on the presentation $\w$ and not only on the variety. For the chart family the exponents are $(q, 1, \dots, 1)$, $d_\phi = 1$, and $\phi([x_0 : \cdots : x_n]) = [x_0^q : x_1 : \cdots : x_n]$ is the well-formalization isomorphism $\WP^n_{(1,q,\dots,q)} \cong \bP^n$; no sparsity can occur because there is nothing to lift. For a well-formed presentation with $q \geq 2$, by contrast, $d_\phi = q^n / \prod q_i > 1$ and \cref{lem:wf-pairs} forces Kummer conditions at every prime. Since both $\h$ and $\hst$ are attached to the presentation $\w$ rather than to the underlying variety, and since presentations arising from moduli problems, such as $(2,4,6,10)$ for genus-two curves, are frequently not well-formed, both regimes occur in practice; the two are related by base change of the height parameter, the well-formalization preserving the exponents $n_i$, so that $Z_{\h,(2,4,6,10)}(X) = Z_{\h,(1,2,3,5)}(X^2)$ (\cref{lem:stratum-height}).

For the tautological height, Deng's theorem confirms the exponent $Q$ predicted by the stacky Batyrev--Manin--Malle framework of \cite{ESZB23}, within which $\hst$ is the height of the tautological bundle; see also \cite{Dar21}. The weighted height is the stable height of that framework (\cref{rem-stack}), for which no counting conjecture is formulated there, and \cref{thm:main} plays the role of the Batyrev--Manin asymptotic for it: the exponent $q\,a(\w)$ is a Fujita-type invariant computed by a polyhedral program, the power of the logarithm is governed not by a Picard rank, which is trivial here since $\Pic(\WP^n_{\w}) \cong \Z$, but by the dimension of the optimal face of that program, that is, by the lattice of Kummer conditions, and the accumulating subvarieties are the coordinate strata singled out by \cref{cor:full-count}. What remains is the extension of \cref{thm:main} beyond $\Q$.

\begin{conj}\label{conj-1}
Let $n \geq 1$, let $\w = (q_0, \dots, q_n)$ with $\gcd(q_0, \dots, q_n) = 1$, and let $k$ be a number field of degree $m$. Then there exist $c = c_{\w}(k) > 0$ and an integer $\beta_k(\w) \geq 0$ such that
\[
Z^{\circ}_{\h}\big( \WP^n_{\w}(k), X \big) \sim c \, X^{m q \, a(\w)} (\log X)^{\beta_k(\w)} ,
\]
where $Z^{\circ}_{\h}$ counts the points with all coordinates nonzero and $Z_{\h}(\WP^n_{\w}(k), X)$ counts points with $\h_k \leq X^m$.
\end{conj}

The exponent should not depend on $k$: it is read off the linear program \cref{eq-LP}, which is a program over $T_{\w}$, and $T_{\w}$ is determined by the exponents $n_i$ alone. The logarithmic power may depend on $k$. For $m > 1$ the conditions defining the lifting set acquire a Selmer-type contribution from $\cO_k^\times$ and the class group of $k$ (\cref{rem-selmer}), and by \cref{prop:number-field} the valuation conditions are then necessary but no longer sufficient; this is why we write $\beta_k(\w)$ rather than $\beta(\w)$, and computing the resulting local densities over $k \neq \Q$ is the first step toward the conjecture.

The conjecture determines the full counting function as well. The argument of \cref{eq-stratification} applies over any $k$ unchanged, the coordinates outside $S$ imposing no condition on $\lambda$, so that restriction to $S$ is a bijection carrying tuples over $k$ to tuples over $k$; hence $Z_{\h}(\WP^n_{\w}(k), X)$ grows like the sum of the lexicographically maximal strata, exactly as in \cref{cor:full-count} over $\Q$, and a proper stratum may dominate.

For $k = \Q$ the conjecture is \cref{thm:main}. For pairwise coprime exponents and arbitrary $k$ it holds as well: there $a(\w) = n+1$ and $\beta(\w) = 0$, and \cref{rem-chart-k} gives $Z_{\h}(\WP^n_{\w}(k), X) = S_k(n) X^{qm(n+1)} + O(X^{qm(n+1)-q}\log X)$, the strata contributing $O(X^{qmn})$. The two confirmations are independent, the first fixing the base field and letting the weights vary, the second fixing the weights and letting the field vary. The bound $q\,a(\w) \geq Q$ of \cref{prop:a-bounds} is consistent with the conjecture over every $k$, since $\h_k \leq \hst_k$ and $\hst$ counts at $X^{mQ}$.

\subsection*{Points of fixed degree}
Let $k(\p)$ be the Galois-theoretic field of definition of $\p \in \WP^n_{\w}(\bar k)$, the fixed field of its stabilizer, and put
\[
\WP^n_{\w}(k; e) = \{ \p \in \WP^n_{\w}(\bar k) : [k(\p) : k] = e \} .
\]
By \cref{prop:fod}, on the interior $k(\p)$ is the minimal field over which $\p$ has a representative, and $\WP^n_{\w}(k;e)$ is the disjoint union over the fields $K$ of degree $e$ over $k$ of the sets of points having a representative over $K$ but over no proper subfield containing $k$. This is the exact analogue of the sets $\bP^n(k; e)$ of Schmidt \cite{schmidt1995} and Widmer \cite{widmer2009}, and the natural refinement of \cref{conj-1} asks for the asymptotic of $Z^{\circ}_{\h}(\WP^n_{\w}(k; e), X)$. Two features of this problem can be seen now. First, $\phi$ maps $\WP^n_{\w}(k;e)$ into $\bigcup_{e' \leq e} \bP^n(k; e')$, and a point $y \in \bP^n(k)$ with nonzero coordinates that does not lift to $\WP^n_{\w}(k)$ contributes its $d_\phi$ preimages to the sets $\WP^n_{\w}(k; e)$ with $e > 1$ according to the Galois orbits on the fiber; so the points of the image that do not lift reappear as points of higher degree of the weighted space: the preimages of the $k$-points of $\bP^n$ with nonzero coordinates are distributed among the sets $\WP^n_{\w}(k;e)$, $1 \leq e \leq d_\phi$, and number $d_\phi \, Z^{\circ}_H(\bP^n(k), X^q)$ in total. Second, for pairwise coprime exponents $d_\phi = 1$ and $\phi$ is a bijection $\WP^n_{\w}(k;e) \to \bP^n(k;e)$ preserving heights up to the exponent $q$, so that, when $e \geq 2$ and $n > 5e/2 + 4 + 2/(me)$, Widmer's theorem \cite{widmer2009} gives $Z_{\h}(\WP^n_{\w}(k;e), X) = D_e(n) X^{qme(n+1)} + O(X^{qme(n+1) - q} \log X)$ with $D_e(n) = \sum_{[K:k] = e} S_K(n)$. For general weights the count of degree-$e$ points requires understanding which Kummer classes in $H^1(k, G)$ have a given splitting behavior, and we leave it open.

\begin{rem}
The conjecture concerns the specific polarization underlying $\h$, namely the pullback of $\cO(1)$ normalized by $1/q$. For a general ample divisor $A$ on a weighted variety $W \subseteq \WP^n_{\w}$ with $\h_A = H_A(\phi(\,\cdot\,))^{1/q}$, the expected exponent involves the Fujita invariant $\alpha(A)$, the minimal $t > 0$ such that $K_W + tA$ is nef \cite{Fuj90}, as in the classical conjecture for toric varieties \cite{Batyrev-Tschinkel-1998, FultonToric}; we do not formulate a precise statement here.
\end{rem}

\begin{rem}
The conjecture is meaningful only for varieties on which rational points are dense. For a smooth curve $C \subset \WP^n_{\w}(k)$ of genus $g \geq 2$, the image $\phi(C)$ is a curve of the same genus in $\bP^n$, so $C(k)$ is finite by Faltings' theorem and $Z_{\h}(C(k), X) = O(1)$; no growth statement is available or needed. On the other hand, for weighted subvarieties arising from moduli problems on which rational points are dense, such as the locus $\cL_2 \subset \WP^3_{(2,4,6,10)}(\Q)$ of genus-two curves with extra automorphisms \cite{2024-03}, height bounds of the type proved here control the arithmetic of the corresponding curves, and the sparsity exponent measures how much of the rational invariant data on the image fails to come from curves with rational coordinates.
\end{rem}

In summary, the rational points of $\WP^n_{\w}$ are as numerous as those of its Veronese image only for pairwise coprime exponents. For every other presentation, and in particular for every well-formed one with $q \geq 2$, the counting exponent of $\h$ is $q\,a(\w)$, strictly below $q(n+1)$ and at least $Q$, with equality at $Q$ exactly under \cref{eq-Q-criterion}, and the power of the logarithm is the dimension of the optimal face of \cref{eq-LP}.


\bibliography{sh-94}

\end{document}